\documentclass[a4paper,reqno]{amsart}

\usepackage[T1]{fontenc}
\usepackage{lmodern}
\usepackage[utf8]{inputenc}
\usepackage{microtype}

\usepackage{mathtools,amssymb,amsthm,graphicx,xcolor,tikz,etoolbox}
\usetikzlibrary{calc}
\usepackage[british]{babel}

\usepackage{bbm}

\usepackage[
	bookmarks=true,
	bookmarksnumbered=true,
	bookmarksopen=true,
	unicode=true,
	pdftoolbar=true,
	pdfmenubar=true,
	pdffitwindow=false,
	pdfstartview={FitH},
	pdftitle={Sylvester's four point problem for ball-convex bodies},
	pdfauthor={Alexandra Bako-Szabo, Florian Besau, Ferenc Fodor},
	pdfsubject={Convex geometry, geometric probability},
	pdfkeywords={Sylvester four-point problem, ball-convexity, spindle convexity,
		R-ball-convex hull, random polytope, Efron-Buchta identity, convex position},
	pdfnewwindow=true,
	colorlinks=true,
	linkcolor=black,
	citecolor=black,
	filecolor=black,
	urlcolor=black]{hyperref}

\usepackage{cleveref}

\numberwithin{equation}{section}
\numberwithin{figure}{section}

\usepackage[nobysame,initials,alphabetic]{amsrefs}
\usepackage{orcidlink}
\usepackage{todonotes}

\allowdisplaybreaks
\ifdefined\pdfsuppresswarningpagegroup \pdfsuppresswarningpagegroup=1 \fi
\newtheorem {theorem}{Theorem}[section]
\newtheorem {proposition}[theorem]{Proposition}
\newtheorem {lemma}[theorem]{Lemma}
\newtheorem {corollary}[theorem]{Corollary}

\newtheorem {conjecture}[theorem]{Conjecture}

\theoremstyle{remark}
\newtheorem {remark}[theorem]{Remark}

\newtheorem {question}[theorem]{Question}

\theoremstyle{definition}

\newcommand{\diam}{\operatorname{diam}}
\newcommand{\dint}{\textup{d}}

\newcommand{\conv}{\operatorname{conv}}
\newcommand{\Li}{\operatorname{Li}}
\newcommand{\artanh}{\operatorname{artanh}}
\newcommand{\as}{\operatorname{as}}

\def\EE{\mathbb{E}}
 \def\NN{\mathbb{N}}
 \def\PP{\mathbb{P}}
 \def\RR{\mathbb{R}}
 \def\SS{\mathbb{S}}
  \def\HH{\mathbb{H}}

\def\bX{\mathbf{X}}
\def\bY{\mathbf{Y}}

 \def\ba{\mathbf{a}}
 \def\bb{\mathbf{b}}

 \def\bm{\mathbf{m}}
 
 \def\bt{\mathbf{t}}
 
 \def\bu{\mathbf{u}}
 \def\by{\mathbf{y}}
 \def\bz{\mathbf{z}}
 \def\bx{\mathbf{x}}
 
 \def\bv{\mathbf{v}}
 \def\be{\mathbf{e}}
 \def\bp{\mathbf{p}}
 \def\bq{\mathbf{q}}
 \def\bc{\mathbf{c}}

\def\cL{\mathcal{L}}

\def\cX{\mathcal{X}}

\def\bone{\mathbbm{1}}

\newcommand{\vol}{\operatorname{Vol}}

\newcommand{\inte}{\operatorname{int}}
\newcommand{\ext}{\operatorname{ext}}
\newcommand{\ibeta}{\mathrm{B}}	

\makeatletter
\let\@fnsymbol\@alph
\makeatother

\begin{document}

\title{\bfseries Sylvester's four point problem for ball-convex bodies}

\author[A. Bak\'o-Szab\'o]{Alexandra Bak\'o-Szab\'o$^{\orcidlink{0009-0009-7652-7179}}$}
\address{Bolyai Institute, University of Szeged, Aradi v\'ertan\'uk tere 1, H-6720 Szeged, Hungary}
\email{szaboa@server.math.u-szeged.hu}

\author[F. Besau]{Florian Besau$^{\orcidlink{0000-0002-6596-6127}}$}
\address{Technische Universität Wien, A-1040 Vienna, Austria}
\email{florian.besau@tuwien.ac.at}

\author[F. Fodor]{Ferenc Fodor$^{\orcidlink{0000-0001-9747-1981}}$}
\address{Bolyai Institute, University of Szeged, Aradi v\'ertan\'uk tere 1, H-6720 Szeged, Hungary}
\email{fodorf@math.u-szeged.hu}

\subjclass{Primary: 52A22, 60D05; Secondary: 52A01, 52A10}
\keywords{Sylvester's four point problem, ball-convexity, spindle convexity,
    $R$-ball-convex hull, random polytope, Efron--Buchta identity,
    convex position, random ball-segment}

\date{}

\begin{abstract}
    We investigate Sylvester's classical four-point problem for ball-convex
    bodies, where convex hulls are replaced by intersections of balls of a
    fixed radius $R$.  We show that the Efron--Buchta identities persist in
    this framework, so that the distribution of the number of vertices of the
    $R$-ball-convex hull of $n$ uniform random points is determined by the
    moments of the volumes of the hulls of its subsamples, and vice versa.

    For three and four uniform random points in an arbitrary planar
    ball-convex body, we determine the Sylvester probabilities of the
    $R$-ball-convex hull.

    We also give an extension of Groemer's inequality for the expected area of the ball-convex hull of $n$ random points and show that it is minimised by the disc of the same area. In particular, this yields an isoperimetric inequality for the ball-convex analogue of the affine length of the boundary curve.

    For the unit disc, we can express the Sylvester probabilities using dilogarithmic functions of the radius $R$, and give explicit values at $R=1$. As $R\to\infty$, these recover the classical Sylvester probabilities.

    As another consequence, we determine the exact
    distribution function of the circumradius of three uniform random points in
    the disc on $[1,\infty)$.
\end{abstract}

\maketitle

\section{Introduction and Main Results}

Sylvester's classical four-point problem asks for the probability that four points chosen independently and uniformly at random from a convex body in the plane are in convex position, that is, that they form the vertices of a convex quadrilateral.
The question was posed by Sylvester in 1864 \cite{Sylvester:1864}, and the answer depends on the shape of the convex body. The problem has been studied extensively and is connected to several areas of mathematics, among them combinatorial geometry, probability theory, and convex geometry. We refer to the surveys \cites{Barany:2008,Calka:2019,Hug:2013,Pfiefer:1989,Reitzner:2010,Schneider:2017} for its history and the results surrounding it.

Let $K\subset \RR^2$ be a convex body in the plane, that is, a compact convex set with non-empty interior, and let $\bX_1,\dotsc,\bX_4$ be independent uniform random points in $K$. Their convex hull
\begin{equation*}
    [K]_4 := \conv\{\bX_1,\dotsc,\bX_4\}
\end{equation*}
is a random polygon in $K$ that, with probability $1$, is either a triangle or a quadrilateral; we write $p_{4,4}(K)$ for the probability of the latter event.
The quantity $p_{4,4}(K)$ is affine invariant, that is, $p_{4,4}(AK+\by)=p_{4,4}(K)$ for every $A\in\mathrm{GL}(2,\RR)$ and $\by\in\RR^2$, where $\mathrm{GL}(2,\RR)$ denotes the group of invertible linear transformations of $\RR^2$.
Blaschke \cite{Blaschke:1917} showed that among all planar convex bodies $p_{4,4}(K)$ is maximal for ellipses and minimal for triangles.
For explicit answers for more than four points in the disc and for regular polygons, we refer to Marckert \cite{Marckert:2017}, Marckert and Rahmani \cite{MR:2021} and Morin \cite{Morin:2025}.

More generally, for a $d$-dimensional convex body $K\subset \RR^d$ and integer $n$, one defines the Sylvester probabilities as
\begin{equation*}
    p_{n,k}(K) := \PP\bigl(f_0([K]_n)=k\bigr),
\end{equation*}
where $[K]_n$ is the convex hull of $n$ independent uniform random points in $K$ and $f_0(P)$ denotes the number of vertices of a polytope $P\subset \RR^d$.
For $n\geq d+1$ the random polytope $[K]_n$ is almost surely $d$-dimensional, hence $p_{n,k}(K)=0$ for $k\leq d$ and $(p_{n,k}(K))_{k=d+1}^{n}$ is the probability distribution of the discrete random variable $f_0([K]_n)$. Extending Efron's classical identity \cite{Efron:1965}, Buchta \cite{Buchta:2005}*{Thm.\ 2} showed that
\begin{equation*}
    p_{n,k}(K) = (-1)^k \binom{n}{k} \sum_{i=d+1}^k (-1)^i \binom{k}{i} \frac{\EE \vol_d([K]_i)^{n-i}}{\vol_d(K)^{n-i}}.
\end{equation*}
Therefore, the family of normalised moments $\EE\vol_d([K]_i)^{j}/\vol_d(K)^{j}$, for $d+1\leq i\leq n$ and $0\leq j\leq n-i$, determines the distribution of $f_0([K]_n)$, and vice versa.
In particular, this yields
\begin{equation*}
    p_{d+2,d+2}(K) = \PP\bigl (f_0([K]_{d+2})=d+2\bigr ) = 1 - (d+2) \frac{\EE\vol_d([K]_{d+1})}{\vol_d(K)}.
\end{equation*}
The numbers $p_{n,k}$ record only how many of the $n$ points are vertices; a
finer question, namely the probability of each combinatorial type of $[K]_n$,
was studied recently by Kabluchko and Panzo \cite{KP:2026}, and Sylvester's
problem for beta-type distributions in every dimension, including uniform points on the sphere, by Gusakova and Kabluchko \cite{GK:2025}.
For the planar Gaussian distribution, the four-point probability equals $\frac{6}{\pi}\arcsin\frac13$, see \cite{FNP:2024} for a recent treatment via Gale duality.

Note that $[K]_{d+1}$ is almost surely a $d$-dimensional simplex. For $d\geq 2$ and $n\geq d+1$ the affine invariant quantity $\EE\vol_d([K]_{n})/\vol_d(K)$ is known to be minimal precisely for ellipsoids, by a theorem of Groemer \cite{Groemer:1974}; the companion statement for the simplex spanned by the origin and $d$ uniform points goes back to Busemann \cite{Busemann:1953}. The maximum of $\EE\vol_d([K]_{d+1})/\vol_d(K)$, on the contrary, is still open for
$d\geq3$: simplices are conjectured to be the unique maximisers.  For $d=2$
this is Blaschke's theorem \cite{Blaschke:1917}, which Campi, Colesanti and
Gronchi \cite{CCG:1999} extended by means of shadow systems to all moments of
$\vol_2([K]_3)$, while \cite{CCG:2001} gives partial results in higher
dimensions by shaking.  Extremal inequalities of this type, in which a
classical isoperimetric-type inequality is strengthened to a stochastic
dominance statement for random convex hulls, are surveyed by Paouris and
Pivovarov \cite{PP:2017b}.

\medskip

In this paper, we investigate Sylvester's four-point problem in the context of ball-convexity, a generalisation of classical convexity which has attracted considerable attention in recent years.
For a bounded set $A\subset\RR^d$ and a radius $R>0$, the \emph{$R$-ball-convex hull} $\conv_R(A)$ is the intersection of all closed balls of radius $R$ containing $A$, and a convex body $K$ is called \emph{$R$-ball-convex} if $\conv_R(K)=K$. This is the exact analogue of the classical description of a convex body as the intersection of all closed half-spaces containing it, with half-spaces replaced by balls of radius $R$. In particular, every $R$-ball-convex body is convex in the classical sense, while the converse fails.

The systematic study of this notion goes back to the strong convexity of
Polovinkin \cite{Polovinkin:1996} and Balashov and Polovinkin \cite{BP:2000},
and to Bezdek, L\'angi, Nasz\'odi and Papez
\cite{BLNP:2007} who concentrate on ball-polyhedra. For expositions see the surveys by Artstein-Avidan and
Florentin \cite{AAF:2025}, Bezdek, L\'angi and Nasz\'odi \cite{BLN:2026}, Martini, Montejano and Oliveros \cite{MMO:2019},
as well as \cite{MM:2022}*{Sec.~2}.  Marynych and Molchanov \cite{MM:2022} work
in the wider setting of the $M$-hull, in which the balls are replaced by
translates of a fixed convex body $M$, a point of view also treated by L\'angi,
Nasz\'odi and Talata \cite{LNT:2013}. Several classical functionals have been
transferred to the ball-convex setting recently: affine surface areas, by
Sch\"utt, Werner and Yalikun \cite{SWY:2025} via floating bodies and by
Artstein-Avidan and Chor \cite{AAC:2025} in the form of a $c$-affine surface
area maximised by the ball; reverse isoperimetric inequalities for
$\lambda$-convex bodies by Drach and Tatarko \cites{DT:2023,DT:2026}, confirming conjectures
of Borisenko on the extremality of the $\lambda$-convex lens and of Bezdek on
the inradius of isoperimetric ball-polyhedra; and randomised isoperimetric
inequalities for the intrinsic volumes of ball-polyhedra, by Paouris and
Pivovarov \cite{PP:2017}.

So far, random ball-convex hulls have been investigated mainly in the asymptotic
regime $n\to\infty$.  In the plane, Fodor, Kevei and V\'igh \cite{FKV:2014}
determined the asymptotic behaviour of the expected number of vertices and of the
expected missed area of random disc-polygons in smooth convex discs, with
variance estimates given by Fodor and V\'igh \cite{FV:2018} and more general
generating bodies treated by Fodor, Papv\'ari and V\'igh \cite{FPV:2020}.  In
higher dimensions, Fodor \cite{Fodor:2020} studied the expected number of
facets of random ball-polytopes in bodies that slide freely in a ball, and
observed that for the unit ball approximated by unit-radius ball-polytopes this
expectation stays bounded as $n\to\infty$.  Marynych and Molchanov
\cite{MM:2022} obtained the limiting distribution of the whole $f$-vector of
the $K$-hull of a uniform sample, without normalisation, and Kabluchko,
Marynych and Molchanov \cite{KMM:2025} extended this to $(K,\HH)$-hulls, in which
rigid motions are replaced by an arbitrary family $\HH$ of invertible affine
maps. By contrast, the present paper concerns the exact distribution of $f_0$
for a fixed small number of points.

\subsection{Efron--Buchta identities and the Sylvester probabilities of the disc}
We fix $0<r\leq R$ and assume that $K\subset \RR^d$ is an $r$-ball-convex body. Since every $r$-ball-convex body is $R$-ball-convex for every $R\geq r$, the $R$-ball-convex hull
\begin{equation*}
    [K]_n^R := \conv_R\{\bX_1,\dotsc,\bX_n\}
\end{equation*}
of $n$ independent uniform random points $\bX_1,\dotsc,\bX_n$ in $K$ satisfies $[K]_n^R\subset K$. We denote by $f_0([K]_n^R)$ the number of vertices of $[K]_n^R$, and we define the Sylvester numbers of the $R$-ball-convex hull by
\begin{equation*}
    p_{n,k}^R(K) := \PP(f_0([K]_n^R)=k)
\end{equation*}
for integers $k$ and $n$ such that $2\leq k\leq n$.

In Section~\ref{sec:Efron-Buchta}, we show that the Efron--Buchta identities \cite{Buchta:2005} also hold for the $R$-ball-convex hull.

\begin{theorem}[Efron--Buchta identities for the $R$-ball-convex hull]\label{thm:R-buchta}
    Let $0<r\leq R$ and let $K\subset\RR^d$ be an $r$-ball-convex body. Then
    for every $n\geq2$,
    \begin{equation}\label{eqn:R-buchta-forward}
        \sum_{k=2}^{n} \binom{n-k}{n-i}\, p_{n,k}^R(K)
        = \binom{n}{i} \frac{\EE\vol_d([K]_{i}^R)^{n-i}}{\vol_d(K)^{n-i}}
        \qquad \text{for $i=2,\dotsc,n$.}
    \end{equation}
    Consequently, for $2\leq k\leq n$,
    \begin{equation}\label{eqn:efron-buchta}
        p_{n,k}^R(K) = (-1)^k \binom{n}{k} \sum_{i=2}^k (-1)^i \binom{k}{i}
        \frac{\EE\vol_d([K]_i^R)^{n-i}}{\vol_d(K)^{n-i}}.
    \end{equation}
\end{theorem}

Theorem~\ref{thm:R-buchta} shows that the normalised moments determine the distribution of $f_0([K]_n^R)$, and vice versa.
Note that in \eqref{eqn:efron-buchta} the sum starts at $i=2$ instead of $i=d+1$ as in the classical case. The reason is that the classical convex hull of at most $d$ points is almost surely a lower-dimensional simplex, whereas the $R$-ball-convex hull of two distinct points in an $r$-ball-convex body $K$ is almost surely an $R$-ball-segment, and hence a $d$-dimensional convex body. Therefore, the expected volume $\EE\vol_d([K]_2^R)$ is strictly positive and contributes to the sum in \eqref{eqn:efron-buchta}.

Moreover, for $r\leq R\leq R'$ we have $[K]_n^R\supset[K]_n^{R'}$, because a ball of
radius $R$ is an intersection of balls of radius $R'$. Consequently, every $R$-extremal point of
$\{\bX_1,\dotsc,\bX_n\}$ (see the precise definition in Section~\ref{sec:background}) is
also $R'$-extremal, so $f_0([K]_n^R)\leq f_0([K]_n^{R'})$.  In fact
$f_0([K]_n^R)\to f_0([K]_n)$ almost surely as $R\to\infty$, and $p_{n,n}^R(K)$
increases monotonically to $p_{n,n}(K)$.

\medskip
Our next main result is the exact distribution of $f_0([K]_n^R)$ for $n=3,4$
and arbitrary planar $r$-ball-convex body $K$.  Let $d=2$, and fix $0<r\leq R$.  For two independent uniform random points
$\bX,\bY$ of $K$, write
\begin{equation*}
    L_K:=\tfrac12\|\bX-\bY\|
\end{equation*}
for the random half-distance between them.
Since $K$ is $R$-ball-convex, the $R$-ball-convex hull
$\conv_R\{\bX,\bY\}$ lies in $K$ and is a congruent copy of the
standard \emph{$R$-ball-segment} $\mathrm{Sg}_R(L_K)$, or \emph{$R$-segment} for
short, the lens-shaped intersection of the two
discs of radius $R$ through $\bX$ and $\bY$, of area
\begin{equation}\label{eqn:segment-2d}
    V_R(\ell):=\vol_2\bigl(\mathrm{Sg}_R(\ell)\bigr)
    =2R^2\arcsin\frac{\ell}{R}-2\ell\sqrt{R^2-\ell^2},
    \qquad 0\leq\ell\leq R.
\end{equation}
The two centres are at distance $\sqrt{R^2-\ell^2}$ from the midpoint of
$\conv\{\bX,\bY\}$. We set
\begin{equation}\label{eqn:def-Delta-R}
    \Delta_R(\ell):=\ell\sqrt{R^2-\ell^2},
    \qquad 0\leq\ell\leq R,
\end{equation}
for the area of the triangle spanned by the two points $\bX, \bY$ and either centre,
see Figure~\ref{fig:segment}.

\begin{figure}[t]
    \centering
    \begin{tikzpicture}[scale=1.05]
        \pgfmathsetmacro{\Rr}{3}
        \pgfmathsetmacro{\gam}{40}
        \pgfmathsetmacro{\El}{\Rr*sin(\gam)}
        \pgfmathsetmacro{\Hh}{\Rr*cos(\gam)}

        \coordinate (X)  at (-\El,0);
        \coordinate (Y)  at ( \El,0);
        \coordinate (cm) at (0,-\Hh);
        \coordinate (cp) at (0, \Hh);

        %\draw[dotted] (Y) --++ (cp);
        \draw[latex-latex] (cp) --++ (Y) node[midway,above] {$\ell$};

        %\draw[dotted] (cp) --++ (Y);
        \draw[latex-latex] (Y) --++ (cp) node[midway,right] {$\sqrt{R^2-\ell^2}$};

        % the triangle spanned by X, Y and the upper centre
        \fill[orange!22] (X) -- (Y) -- (cp) -- cycle;

        % the ball-segment
        \fill[blue!25,fill opacity=0.75]
        (Y) arc[start angle=90-\gam,  end angle=90+\gam,  radius=\Rr]
        arc[start angle=270-\gam, end angle=270+\gam, radius=\Rr] -- cycle;

        % chord and axis
        \draw[thin,gray] (X) -- (Y);
        \draw[thin,gray] (cm) -- (cp);

        % hints of the two bounding circles
        \draw[gray!55,dashed,thin]
        ([shift=({70-\gam}:\Rr)]cm) arc[start angle=70-\gam, end angle=110+\gam, radius=\Rr];
        \draw[gray!55,dashed,thin]
        ([shift=({250-\gam}:\Rr)]cp) arc[start angle=250-\gam, end angle=290+\gam, radius=\Rr];

        % boundary of the ball-segment
        \draw[thick]
        (Y) arc[start angle=90-\gam,  end angle=90+\gam,  radius=\Rr]
        arc[start angle=270-\gam, end angle=270+\gam, radius=\Rr] -- cycle;

        % the kite
        \draw[thin,dashed] (X) -- (cp) -- (Y);
        \draw[thin,dashed] (X) -- (cm) -- (Y);

        % right angle at the midpoint
        \draw[thin,gray] (0.25,0) -- (0.25,0.25) -- (0,0.25);

        % the angle gamma at the lower centre
        \draw[thin] ([shift=({-90+\gam}:0.9)]cp) arc[start angle=-90+\gam, end angle=-90, radius=0.9];
        \node at ([shift=({-90+\gam/2}:0.7)]cp) {$\gamma$};

        % points and labels
        \fill (X) circle (1.6pt) node[left]  {$\bx$};
        \fill (Y) circle (1.6pt) node[right] {$\by$};
        \fill (cm) circle (1.6pt) node[below] {$\bc_-$};
        \fill (cp) circle (1.6pt) node[above] {$\bc_+$};
        %\fill (0,0) circle (1.2pt);

        %\node[below,inner sep=2pt] at (\El/2,-0.04) {$\ell$};
        %\node[rotate=90,anchor=south,inner sep=3pt] at (0,0.62*\Hh) {$\sqrt{R^2-\ell^2}$};
        \node[anchor=south west,inner sep=1pt] at ($(cp)!0.5!(Y)+(0.03,0)$) {$R$};

        % callouts
        %\node at (0.5,1.0) {$\Delta_R(\ell)$};
        %\draw[thin,-latex] (2.5,1.72) -- (0.95,0.95);
        \node at (-0.5,-0.3) {$\mathrm{Sg}_R(\ell)$};
        %\draw[thin,-latex] (-2.5,1.55) -- (-0.8,0.32);

        \draw (\El,\Hh) ++ (-0.1,-0.1) --++ (0.2,0.2);
        \draw (\El,\Hh) ++ (0.1,-0.1) --++ (-0.2,0.2);

    \end{tikzpicture}
    \caption{The $R$-segment $\mathrm{Sg}_R(\ell)=\conv_R\{\bx,\by\}$ of two points
        at half-distance $\ell$ and the triangle $\conv\{\bx,\by,\bc_+\}$.
    }
    \label{fig:segment}
\end{figure}

By the Efron--Buchta identities, the distribution of $f_0([K]_3^R)$ is
\begin{equation}\label{eqn:pnk3}
    \begin{aligned}
        p_{3,2}^R(K) & = 3 \frac{\EE \vol_2([K]_2^R)}{\vol_2(K)} = \frac{3}{\vol_2(K)} \EE V_R(L_K), \\
        p_{3,3}^R(K) & = 1- p_{3,2}^R(K)= 1- \frac{3}{\vol_2(K)} \EE V_R(L_K),                       \\
    \end{aligned}
\end{equation}
and  the distribution of $f_0([K]_4^R)$ is
\begin{equation}\label{eqn:pnk}
    \begin{aligned}
        p_{4,2}^R(K) & = 6 \frac{\EE \vol_2([K]_2^R)^2}{\vol_2(K)^2} = \frac{6}{\vol_2(K)^2} \EE V_R(L_K)^2, \\
        p_{4,3}^R(K) & = 4 \frac{\EE \vol_2([K]_3^R)}{\vol_2(K)} - \frac{12}{\vol_2(K)^2} \EE V_R(L_K)^2,    \\
        p_{4,4}^R(K) & = 1-p_{4,2}^R(K)-p_{4,3}^R(K).
    \end{aligned}
\end{equation}
Hence, the law of $L_K$ together with $\EE\vol_2([K]_3^R)$ completely determines the distribution.
It turns out that $\EE\vol_2([K]_3^R)$ can be expressed in terms of the moments of functions of $L_K$ and the classical Sylvester functional $\EE\vol_2([K]_3)$.

\begin{theorem}[Sylvester probabilities for the $R$-ball-convex hull in the plane]\label{thm:sylvester-general}
    Let $0<r\leq R$, let $K\subset\RR^2$ be an $r$-ball-convex body.
    Then
    \begin{align}\notag
        \EE\vol_2\bigl([K]_3^R\bigr)
         & = \EE\vol_2\bigl([K]_3\bigr) + \frac32\,\EE V_R(L_K)     \\ \label{eqn:intro-vol3}
         & \quad + \frac{1}{\vol_2(K)}\Bigl(\frac94\,\EE V_R(L_K)^2
        + 12\,\EE\bigl[\Delta_R(L_K)V_R(L_K)\bigr] - 16\,\EE L_K^4\Bigr),
    \end{align}
    and therefore
    \begin{align*}
        p_{4,3}^R(K) & = \frac{4}{\vol_2(K)}\,\EE\vol_2\bigl([K]_3\bigr) + \frac{6}{\vol_2(K)}\,\EE V_R(L_K) \\
                     & \qquad + \frac{1}{\vol_2(K)^2}\Bigl(48\,\EE\bigl[\Delta_R(L_K)V_R(L_K)\bigr]
        - 3\,\EE V_R(L_K)^2 - 64\,\EE L_K^4\Bigr).
    \end{align*}
\end{theorem}

Theorem~\ref{thm:sylvester-general} identifies the dependence on the body $K$ in
two ingredients of very different character: the law of the half-distance 
$L_K$, which contains the two-point dependent quantities $\EE V_R(L_K)$, $\EE V_R(L_K)^2$,
$\EE[\Delta_R(L_K)V_R(L_K)]$ and $\EE L_K^4$, and the classical Sylvester
functional $\EE\vol_2([K]_3)$, which enters only the leading term of
$p_{4,3}^R$.  The radius $R$ enters only through
$V_R$ and $\Delta_R$.  Letting $R\to\infty$ one has $V_R(\ell)\to0$ and
$\Delta_R(\ell)V_R(\ell)\to\tfrac43 \ell^4$, so that the last two terms of
\eqref{eqn:intro-vol3} cancel and \eqref{eqn:intro-vol3} recovers
$\EE\vol_2([K]_3)$.

\medskip
For the disc all of the functionals in \eqref{eqn:pnk} are explicit by Theorem \ref{thm:sylvester-general}, because the law of $L_{B^2}$
is classical and has an explicit density function, see \eqref{eqn:dist-density-2d}, while
$\EE\vol_2([B^2]_3)=\frac{35}{48\pi}$ is Sylvester's
classical value. Thus, everything involving $L_{B^2}$ reduces to
one-dimensional integrals, which we evaluate in closed form in
Section~\ref{sec:moments} and explicitly for $R=1$.

\begin{theorem}[Sylvester probabilities for the disc $B^2$]\label{thm:sylvester-r-convex}
    Let $R\geq 1$.
    Then the
    distributions of $f_0([B^2]_3^R)$ and $f_0([B^2]_4^R)$ are determined by the
    three moments
    \begin{equation*}
        m_1(R) = \EE V_R(L_{B^2}),\qquad
        m_2(R) = \EE V_R(L_{B^2})^2,\qquad
        \mu(R) = \EE\bigl[\Delta_R(L_{B^2})\,V_R(L_{B^2})\bigr].
    \end{equation*}
    In particular, at the smallest admissible radius $R=1$, we have
    \begin{align*}
        \bigl(p_{3,2}^1(B^2),\;p_{3,3}^1(B^2)\bigr)
         & = \Bigl(\frac{32}{\pi^2}-3,\; 4-\frac{32}{\pi^2}\Bigr)
        \approx ( 24.23\%,\; 75.77\%),                                              \\
        \bigl(p_{4,2}^1(B^2),\;p_{4,3}^1(B^2),\;p_{4,4}^1(B^2)\bigr)
         & =\Bigl(\frac{81}{2\pi^2}-4,\;\frac{6}{\pi^2},\;5-\frac{93}{2\pi^2}\Bigr) \\
         & \approx (10.35\%,\;60.79\%,\;28.86\%),
    \end{align*}
    while as $R\to\infty$,
    \begin{equation}\label{eqn:intro-p44-rate}
        p_{4,4}^R(B^2) = p_{4,4}(B^2)-\frac{2048}{525\pi^2R}+O\!\bigl(R^{-2}\bigr),
        \quad p_{4,4}(B^2)=1-\frac{35}{12\pi^2}\approx 70.45\% .
    \end{equation}
\end{theorem}

We note that none of the three moments is elementary for $R>1$. However, each of them
evaluates in terms of $\artanh\frac1R$ and dilogarithms of $\pm\frac1R$, so
that all probabilities of Theorem~\ref{thm:sylvester-r-convex} admit a closed
form.  This is Proposition~\ref{prop:dilog}. The
elementary values at $R=1$ come from the evaluations
$\Li_2(1)=\frac{\pi^2}{6}$ and $\Li_2(-1)=-\frac{\pi^2}{12}$.

\subsection{A Groemer-type inequality and the extremiser of the Sylvester probabilities}

In Subsection~\ref{sec:groemer}, we use Steiner symmetrisation on $r$-ball-convex bodies
as recently investigated by Artstein-Avidan and Florentin \cite{AAF:2026}. Their tools allow us to derive the following extension of Groemer's inequality.

\begin{theorem}[Groemer-type inequality for the ball-convex hull in the plane]\label{thm:R-groemer}
    Let $K\subset \RR^2$ be an $r$-ball-convex body and let $R\geq r$. For $n\geq 2$ and $k\geq 1$,
    \begin{equation*}
        \EE\vol_2([K]_n^R)^k \geq \EE\vol_2([r_KB^2]_n^R)^k,
    \end{equation*}
    where $r_K:=\sqrt{\frac{1}{\pi}\vol_2(K)}$. Equality holds if and only if $K$ is a disc.
\end{theorem}

The case $n=2$ can also be extended to all dimensions $d\geq 2$ using a result of Pfiefer \cite{Pfiefer:1990}, see Corollary \ref{cor:isoperimetric}.

By Theorem \ref{thm:R-groemer} $p_{4,2}^R(K)$ is minimised and
$\EE f_0([K]_4^R)$ is maximised by $r_K B^2$.
However, unlike in the classical case, $\EE f_0([K]_4^R)$ does not determine the behaviour of $p_{4,4}^R(K)$.

We conjecture that $p_{4,4}^R(K)$ is maximal for the disc among all $r$-ball-convex bodies of fixed area, see Conjecture \ref{conj:p43-p44}. Although the full inequality seems out of reach, we are able to prove the following.

\begin{theorem}[the disc as an extremiser of $p_{4,4}^R$]\label{thm:intro-p43}
    Let $K\subset \RR^2$ be an $r$-ball-convex body and let $R\geq r >0$. If
    \begin{equation}\label{eqn:small-body}
        (\diam K)^3 \leq \frac{8}{\pi} \vol_2(K)\,R,
    \end{equation}
    then
    \begin{equation*}
        p_{4,4}^R(K) \leq p_{4,4}^R(r_K B^2),
    \end{equation*}
    with equality if and only if $K$ is a disc.
    In particular, \eqref{eqn:small-body} holds true for any convex body $K\subset \RR^2$ of constant width $r$.
\end{theorem}

As a by-product of our methods we obtain two immediate corollaries which we briefly outline in the following.

\subsection{Isoperimetric inequality for the ball-convex analogue of affine surface area}

First, if $K\subset \RR^2$ is an $r$-ball-convex body of class $C^2_+$, then Fodor, Kevei and V\'{\i}gh \cite{FKV:2014} showed that
\begin{equation*}
    \lim_{n\to\infty} \Bigl(\frac{n}{\vol_2(K)}\Bigr)^{\frac{2}{3}}\, \Bigl(\vol_2(K) - \EE\vol_2([K]_n^r) \Bigr)
    = c_2 \as^r(K),
\end{equation*}
where $c_2 := (2/3)^{1/3}\, \Gamma(5/3)$ and $\as^r(K)$ is the $r$-ball-convex analogue of the affine surface area,  see Subsection \ref{sec:c-affine}. We note that $\as^r$ was recently extended to all dimensions $d\geq2$ for all $r$-ball-convex bodies by Schütt, Werner and Yalikun \cite{SWY:2025}.

For $n\to\infty$, Theorem \ref{thm:R-groemer} therefore yields the following isoperimetric inequality in the limit.

\begin{corollary}[isoperimetric inequality for $\as^r(K)$ in the plane]\label{cor:R-affine-iso}
    Let $K\subset \RR^2$ be an $r$-ball-convex body of class $C^2_+$. Then
    \begin{equation}\label{eqn:R-affine-iso}
        \as^r(K) \leq \as^r(r_K B^2) = 2\pi\, r_K^{\frac{2}{3}} \Bigl(1-\frac{r_K}{r}\Bigr)^{\frac{1}{3}}.
    \end{equation}
\end{corollary}
Corollary~\ref{cor:R-affine-iso} implies, in particular, the isoperimetric inequalities derived previously in \cite{AAC:2025} and \cite{SWY:2025} in the case $d=2$ for sufficiently smooth $r$-ball-convex bodies, and for $r\to\infty$ it yields the classical affine isoperimetric inequality.

\subsection{Tail-distribution of the circumradius of three random points}

The second corollary of our methods follows as a by-product of Lemma~\ref{lem:structure}, which identifies the event
$\{f_0([K]_3^R)=2\}$ with the purely metric event $\{\varrho([K]_3)>R\}$, where $\varrho$ denotes the circumradius of the triangle $[K]_3$.
Together with \eqref{eqn:pnk3}, it therefore answers a question which is unrelated to ball-convexity: it determines the distribution of the circumradius
$\varrho([K]_3)$ on $[r,\infty)$, and explicitly for the disc by
Theorem~\ref{thm:sylvester-r-convex}.

\begin{corollary}[circumradius of three random points]\label{cor:intro-circumradius}
    Under the same hypotheses as in Theorem~\ref{thm:sylvester-general}, we have
    \begin{equation*}
        \PP(\varrho([K]_3)>R) = p_{3,2}^{R}(K)
        \qquad\text{for every } R \geq r,
    \end{equation*}
    and therefore, by Theorem \ref{thm:R-groemer},
    \begin{equation*}
        \PP(\varrho([K]_3)>R) \geq \PP\biggl(\varrho([B^2]_3)>\frac{R}{r_K}\biggr).
    \end{equation*}
    Moreover
    \begin{align*}
        \PP(\varrho([B^2]_3)>1) & = \frac{32}{\pi^2}-3 \approx 24.2278\% ,                                      \\
        \PP(\varrho([B^2]_3)>R) & = \frac{1024}{525\pi^2R}+O\!\bigl(R^{-3}\bigr) \quad \text{for $R\to\infty$}.
    \end{align*}
\end{corollary}

The second part of Corollary~\ref{cor:intro-circumradius} is the case $d=2$ of Theorem~\ref{thm:circumradius-d} below.
The circumradius of a uniform
sample has been investigated by Affentranger \cite{Affentranger:1989}, who showed that the
probability that the circumdisc of $[B^2]_3$ is contained in $B^2$ is $\tfrac25$.
Le~Ca\"er \cite{LeCaer:2017} determined the distribution of $\varrho([B^2]_3)$
conditioned on the same containment event.  The unconditional distribution was
examined in \cite{LeCaer:2017} only by Monte Carlo simulation.
Corollary~\ref{cor:intro-circumradius} confirms the decay rate conjectured on
that basis and identifies the constant, $\frac{1024}{525\pi^2}\approx0.19762$,
within the confidence interval $0.197\pm0.001$ reported in \cite{LeCaer:2017}*{Sec.~VIII}.

\subsection{Outline of the paper}

The paper is organised as follows.  In Section~\ref{sec:background}, we collect the
necessary background on $R$-ball-convexity and in Section~\ref{sec:Efron-Buchta} we prove
Theorem~\ref{thm:R-buchta}. Section~\ref{sec:two} is devoted to two random
points: we compute the moments $\EE\vol_d([rB^d]_2^R)^k$ of the volume of a
random ball-segment in every dimension and deduce the probabilities $p_{n,2}^R$.
In Section~\ref{sec:three}, we analyse the
combinatorial structure of the $R$-ball-convex hull of three points, deduce
the distribution of the circumradius of three random points in $B^d$ in every
dimension, and prove Theorems~\ref{thm:sylvester-general}
and~\ref{thm:sylvester-r-convex} up to the evaluation of three one-dimensional
moments. These are evaluated in Section~\ref{sec:moments}, first as series with
rational coefficients and then in closed form in terms of $\artanh\frac1R$ and
dilogarithms.  In Section~\ref{sec:extremal}, we investigate the extremal problem
in the plane and prove Theorem~\ref{thm:R-groemer} and Theorem~\ref{thm:intro-p43}.
In the final Section~\ref{sec:discussion}, we list some open questions for future investigation.

\section{Background on \texorpdfstring{$R$}{R}-ball-convexity}\label{sec:background}

We denote the Euclidean inner product in $\RR^d$ by $\langle\cdot,\cdot\rangle$ and the norm by $\|\cdot\|$. We set
\begin{equation*}
    B^d(\bc,\rho) := \{\bx\in\RR^d : \|\bx-\bc\|\leq\rho\}
\end{equation*}
for the closed ball with centre $\bc\in\RR^d$ and radius $\rho>0$, and
write $B^d:=B^d(\mathbf{0},1)$ for the Euclidean unit ball, whose volume is
$\kappa_d:=\vol_d(B^d)$.  We write $\inte$, $\partial$ and $\conv$ for the
interior, the boundary, and the convex hull, respectively, and $\vol_d$ for the $d$-dimensional
Lebesgue measure.  Throughout the paper, we assume $d\geq2$ without further mentioning it.  For standard facts from the theory of
convex bodies, we refer to Schneider's monograph \cite{Sch14}, and for information on random polytopes to the book \cite{SW:2008} by Schneider and Weil.

We recall the basic definitions of $R$-ball-convexity and fix our notation
against the systematic accounts of the theory.  Artstein-Avidan and Florentin
\cite{AAF:2025} study the class $\mathcal{S}^d$ of \emph{ball-bodies}, which are
intersections of translates of the unit ball, together with the associated
$c$-duality $A\mapsto A^{c}$. Bezdek, L\'angi and Nasz\'odi \cite{BLN:2026}
survey the theory of $r$-ball bodies, concentrating especially on their volumetric properties, the combinatorial structure of ball-polyhedra, best and random approximations, and some related problems, such as the Kneser--Poulsen conjecture.  Marynych and Molchanov
\cite{MM:2022}*{Sec.~2} work in the more general setting of the \emph{$M$-hull}
$\conv_M(A)$, which is the intersection of all translates of a fixed convex body $M$
containing $A$, and of the corresponding $M$-strongly convex sets in the sense
of \cites{Polovinkin:1996,BP:2000}, a point of view also discussed in detail by L\'angi,
Nasz\'odi and Talata \cite{LNT:2013}. The paper \cite{BLN:2026} also reviews results on $M$-convex bodies under the name $L$-convexity.  Our terminology mostly follows Bezdek et al. \cite{BLN:2026}. In this dictionary, our $\conv_R$ is the
$M$-hull for $M=RB^d$, and $\conv_1$ is the $c$-hull of
\cite{AAF:2025}*{Def.~1.2}.

We note that the origins of $r$-ball convexity go back to at least the 1930s. The concept appeared in the paper \cite{M:1935} under the name \emph{\"Uberkonvexit\"at} in the Minkowski geometry setting. In the decades following Mayer's paper $r$-ball-convexity was reintroduced several times under different names. For the early history of the subject, we refer to the survey \cite{BLN:2026} by Bezdek et al.\ and the book by Martini, Montejano and Oliveros \cite{MMO:2019} and the references therein.

Since Euclidean balls are generating sets in the sense of
\cite{Polovinkin:1996}, see also \cite{LNT:2013}, $R$-ball-convexity coincides
with the \emph{$R$-spindle convexity} of \cite{BLNP:2007}. A set is called $r$-spindle convex if along with any two of its points, the $r$-ball-segment they span is also contained in the body. The $r$-ball-segment is the intersection of all radius $r$ closed balls containing the two points, and it is also called the \emph{$r$-spindle} (see
\cite{BLNP:2007}).  For a general convex body $M$ the two notions differ: $M$-strong
convexity can be strictly stronger than $M$-spindle convexity. The two concepts are
equivalent exactly when $M$ is a generating set, see \cite{MM:2022}*{p.~8},
\cite{LNT:2013} and \cite{AAF:2025}*{Prop.~1.5}.

Let $R>0$ and let $A\subset\RR^d$ be bounded.  The \emph{$R$-ball-convex
    hull} of $A$ is
\begin{equation}\label{eqn:convR-def}
    \conv_R(A) := \bigcap\{B^d(\bc,R) : \bc\in\RR^d,\ A\subset B^d(\bc,R)\},
\end{equation}
the intersection of all closed balls of radius $R$ containing $A$. If such a ball does not exist, then
$\conv_R(A):=\RR^d$.  A convex body
$K\subset\RR^d$ is \emph{$r$-ball-convex} if $K=\conv_r(K)$, that is, if
$K$ is equal to the intersection of closed balls of radius $r$ containing it.

If $A=\{\bx_1,\dotsc,\bx_m\}$ is finite, a point $\bx_i\in A$ is
\emph{$R$-ball-extremal}, or briefly \emph{$R$-extremal}, if
$\bx_i\notin\conv_R(A\setminus\{\bx_i\})$.  We call the $R$-extremal points
of $A$ the \emph{vertices} of $\conv_R(A)$ and write $f_0(\conv_R A)$ for
their number. If, moreover, $A$ lies in an open ball of radius $R$, then by
Lemma~\ref{lem:R-minkowski} the set of $R$-extremal points is the
unique smallest subset of $A$ whose $R$-ball-convex hull is equal to
$\conv_R(A)$, so in this case $f_0$ is well defined.

If $\bx_i$ is not $R$-extremal, then $\bx_i\in\conv_R (A\setminus \{\bx_i\})$, and we say that $\bx_i$ is \emph{absorbed} by $A\setminus \{\bx_i\}$.

For $R=1$, the $R$-extremal points of $A$ are the $c$-extreme points $\ext^{c}$
of \cite{AAF:2025}*{Sec.~4}, and hence, after the rescaling of
Lemma~\ref{lem:elementary}\,(b), for every $R>0$. The quantity $f_0$ is the
zeroth entry of the $f$-vector of an $M$-strongly convex set introduced in
\cite{MM:2022}*{Def.~3.3 and Def.~4.4}, specialised to $M=RB^d$.

\begin{lemma}[elementary properties of $\conv_R$]\label{lem:elementary}
    Let $R\geq r>0$ and let $A,A'\subset\RR^d$ be bounded.
    \begin{enumerate}
        \item[(a)] $A\subset\conv_R(A)\subset \conv_r(A)$;
              $A\subset A'$ implies $\conv_R(A)\subset\conv_R(A')$; and
              $\conv_R(\conv_R(A))=\conv_R(A)$.
        \item[(b)] For every rigid motion $g$ of $\RR^d$ and every $t>0$,
              \begin{equation}\label{eqn:equivariance}
                  \conv_R(gA)=g\,\conv_R(A)
                  \qquad\text{and}\qquad
                  \conv_{t R}(t A)=t\,\conv_R(A).
              \end{equation}
              In particular only the ratio of the radius to the scale of $A$
              matters, and either of the two may be normalised.
        \item[(c)] If $K\subset\RR^d$ is $r$-ball-convex and $A\subset K$, then $\conv_R(A)\subset K$.
    \end{enumerate}
\end{lemma}

All three statements are immediate from \eqref{eqn:convR-def} and are known:
see \cite{BLNP:2007}, \cite{BLN:2026}, \cite{AAF:2025}*{Sec.~1} and \cite{MM:2022}*{Sec.~2}.

The operation of the hull $\conv_R$ is compatible with the taking of sections with affine subspaces.  We use the
following elementary fact to transfer planar statements to
higher dimensions.

\begin{lemma}\label{lem:slicing}
    Let $E\subset\RR^d$ be an affine subspace with $\dim E\geq1$, and let
    $A\subset E$ be a non-empty bounded set contained in some ball of radius
    $R$.  Write $\conv_R^{E}$ for the hull operation \eqref{eqn:convR-def}
    performed inside $E$, with balls of radius $R$ of $E$.  Then
    \begin{equation*}%\label{eqn:slicing}
        \conv_R(A)\cap E = \conv_R^{E}(A) .
    \end{equation*}
\end{lemma}

This is \cite{BLNP:2007}*{Prop.~5.2}, stated there for the spindle convex hull
of a set of circumradius at most $R$, which is the same $\conv_R$ as $B^d$ is a generating set; see also \cite{BLNP:2007}*{Cor.~3.4}.

Note that Lemma~\ref{lem:slicing} yields that absorption among points in
a common affine subspace is decided inside that subspace: if
$A\cup\{\bx\}\subset E$, then $\bx\in\conv_R(A)$ if and only if
$\bx\in\conv_R^{E}(A)$.

It is also a consequence of Lemma~\ref{lem:slicing} that the $R$-ball-segment
$\conv_R\{\bx,\by\}\subset\RR^d$ is a body of revolution that is rotationally symmetric about the line containing the endpoints.

The following duality related to $\conv_R$ has been repeatedly introduced.  For a
bounded set $A\subset\RR^d$, let
\begin{equation*}
    C_R(A) := \bigcap_{\ba\in A} B^d(\ba,R)
    = \{\bc\in\RR^d : A\subset B^d(\bc,R)\}.
\end{equation*}
The elements of $C_R(A)$ are also called \emph{admissible centres}.
We note that $C_R(A)$ is the $c$-dual $A^{c}$ in the language of
\cite{AAF:2025}*{Def.~1.2} for $R=1$, and the set $\mathrm{cn}_M(A)$ of
\cite{MM:2022}*{Sec.~2} for $M=RB^d$.  As with $\conv_R^{E}$ in
Lemma~\ref{lem:slicing}, we write $C_R^{E}$ for the $C_R(\cdot)$ operation performed
in an affine subspace $E$, so that $C_R^{E}(A)=C_R(A)\cap E$ for
$A\subset E$.

\begin{lemma}
    \label{lem:polarity}
    Let $A\subset\RR^d$ be bounded.  Then $A\mapsto C_R(A)$ is inclusion
 reversing and
    \begin{equation}\label{eqn:polarity}
        \begin{split}
            \conv_R(A)
             & = \!\!\!\!\!\bigcap_{\bc\in C_R(A)}\!\!\!\!\! B^d(\bc,R)
            = \bigl\{\bx\in\RR^d : C_R(A)\subset B^d(\bx,R)\bigr\}
            = C_R\bigl(C_R(A)\bigr).
        \end{split}
    \end{equation}
    Moreover, for $\bx\in A$,
    \begin{equation}\label{eqn:extremal-iff}
        \bx \text{ is not $R$-extremal in } A
        \iff C_R(A\setminus\{\bx\}) = C_R(A).
    \end{equation}
\end{lemma}

The identity \eqref{eqn:polarity} is \cite{AAF:2025}*{Rem.~1.3} and
\cite{MM:2022}*{(2.2) and Prop.~2.2} for $M=RB^d$; see also \cite{BLNP:2007}*{Cor.~3.4}.
Its three descriptions are reformulations of the condition
$\|\bx-\bc\|\leq R$, read once in $\bx$ and once in $\bc$, and
\eqref{eqn:extremal-iff} follows from them by writing
$C_R(A)=C_R(A\setminus\{\bx\})\cap B^d(\bx,R)$.

Thus, $\bx_i$ is $R$-extremal in $A=\{\bx_1,\dotsc,\bx_m\}$ precisely when the
ball $B^d(\bx_i,R)$ is irredundant in the intersection $C_R(A)$.  Deleting one
non-$R$-extremal point of $A$ therefore leaves $\conv_R(A)$ unchanged. Lemma~\ref{lem:R-minkowski} expresses the fact that one can delete all non-$R$-extremal points simultaneously without changing the $R$-ball-convex hull.

\begin{lemma}[$R$-ball-convex Minkowski theorem for finite sets]\label{lem:R-minkowski}
    Let $A=\{\bx_1,\dotsc,\bx_m\}\subset\RR^d$ consist of pairwise distinct
    points with $\inte C_R(A)\neq\emptyset$, and let $S\subset A$ be a set of
    points none of which is $R$-extremal in $A$.  Then
    \begin{equation*}
        C_R(A\setminus S) = C_R(A)
        \qquad\text{and consequently}\qquad
        \conv_R(A\setminus S) = \conv_R(A).
    \end{equation*}
\end{lemma}

Note that the assumption $\inte C_R(A)\neq\emptyset$ means that the circumradius of $A$ is strictly smaller than
$R$.

For general ball-bodies this belongs to the Krein--Milman- and Carath\'eodory-type
results of Artstein-Avidan and Florentin, see
\cite{AAF:2025}*{Cor.~4.19, Thms.~4.21 and 4.22}, see also Bezdek et al. \cite{BLNP:2007}*{Section~5}.  However, we shall only use this
version for finite sets.

In view of \eqref{eqn:extremal-iff}, the first statement is the uniqueness of the
reduced family of generating balls of the ball-polyhedron $C_R(A)$, see
\cite{BLNP:2007}*{Def.~6.3}. The second statement follows from
Lemma~\ref{lem:polarity}.

In the plane, the duality of Lemma~\ref{lem:polarity} takes a specific form, which follows directly from
\eqref{eqn:polarity} and \eqref{eqn:extremal-iff}:
Assume in $\RR^2$ that $C_R(A)\neq\emptyset$.  Call
$\bv\in C_R(A)$ a \emph{vertex} of $C_R(A)$ if $\|\bv-\bx_i\|=\|\bv-\bx_j\|=R$ for two
indices $i\neq j$, and let $\bv_1,\dotsc,\bv_k$ be the vertices of
$C_R(A)$.  If $k\geq2$, then
\begin{equation}%\label{eqn:f0-vertices}
    f_0(\conv_R A) = k \label{lem:ball-polygon}
    \qquad\text{and}\qquad
    \conv_R(A) = \bigcap_{\ell=1}^{k} B^2(\bv_\ell,R) .
\end{equation}
We will use \eqref{lem:ball-polygon} in Section~\ref{sec:three}.

Combining Lemma~\ref{lem:slicing} with duality \eqref{eqn:polarity} produces a
metric obstruction, which is one of the most important geometric ingredients of our argument.

\begin{proposition}[absorption forces large circumradius]\label{prop:pole}
    Let $3\leq n\leq d+1$, let $A=\{\bx_1,\dotsc,\bx_n\}\subset\RR^d$ be affinely
    independent, and let $\varrho$ be the circumradius of
    $\conv\{\bx_1,\dotsc,\bx_n\}$ within its affine hull.  If
    $\bx_j\in\conv_R\bigl(\{\bx_1,\dotsc,\bx_n\}\setminus\{\bx_j\}\bigr)$
    for some $j$, then $\varrho\geq R$.  Equivalently: if $\varrho<R$, then
    all $n$ points are $R$-extremal and
    $f_0(\conv_R\{\bx_1,\dotsc,\bx_n\})=n$.
\end{proposition}

\begin{proof}
    Since $\bx_1,\dotsc,\bx_n$ are affinely independent points, $\conv(A)$ is an $(n-1)$-dimensional simplex in the affine hull of $A$. This simplex has a unique circumcentre $\bz$ and circumradius $\varrho$, such that $\|\bz-\bx_i\|=\varrho$ for all $i\in \{1,\dotsc,n\}$.

    We prove the contrapositive statement, so assume that $\varrho<R$.
    We fix $j\in\{1,\dotsc,n\}$ arbitrarily and show that there exists a ball $B^d(\bc, R)$ such that $A\setminus \{\bx_j\} \subset B^d(\bc,R)$ and $\bx_j\not\in B^d(\bc,R)$, which yields $\bx_j\not\in \conv_R(A\setminus\{\bx_j\})$ and therefore $\bx_j$ is $R$-extremal in $A$. Since $j\in \{1,\dotsc,n\}$ was arbitrary, this shows $f_0(\conv_R(A)) = n$.

    Set
    \begin{equation*}
        \beta := \frac{1}{\varrho^2} \, \max_{i\neq j} \Bigl\{ \langle\bz-\bx_j, \bz-\bx_i\rangle\Bigr\} \in [-1,1].
    \end{equation*}
    Note that $\beta<1$, since otherwise there would be an index $i\neq j$ such that $\bx_i=\bx_j$.
    Since $(R/\varrho)^2-(1-\beta^2)\geq (R/\varrho)^2 - 1 >0$, we set
    \begin{equation*}
        \bc:=t\bz + (1-t)\, \bx_j, \quad \text{for}\quad t:= (1-\beta) + \sqrt{(R/\varrho)^2-(1-\beta^2)}.
    \end{equation*}
    We claim that $t>R/\varrho$. Indeed, if $(1-\beta) > R/\varrho$, then $t\geq (1-\beta)> R/\varrho$, and if $(1-\beta)\leq R/\varrho$, then
    \begin{equation*}
        t > \frac{R}{\varrho}
        \; \Leftrightarrow \; \frac{R^2}{\varrho^2} - (1-\beta^2) > \Bigl(\frac{R}{\varrho} - (1-\beta)\Bigr)^2
        \; \Leftrightarrow \; -(1+\beta) > (1-\beta) - 2\frac{R}{\varrho}
        \; \Leftrightarrow \; R > \varrho.
    \end{equation*}
    Using $t>R/\varrho >1$, we find $\|\bc-\bx_j\| = t \varrho > R$ so $\bx_j\notin B^d(\bc,R)$, and
    for all $i\neq j$,
    \begin{align*}
        \|\bc-\bx_i\|^2
         & = \|(\bz-\bx_i) + (t-1)(\bz-\bx_j)\|^2                                \\
         & = \varrho^2 (1+(t-1)^2) + 2(t-1)\, \langle\bz-\bx_i, \bz-\bx_j\rangle \\
         & \leq \varrho^2\, \Bigl( 1+(t-1)^2+2(t-1)\beta\Bigr)
        = R^2,
    \end{align*}
    which shows that $\bx_i\in B^d(\bc,R)$.
\end{proof}

\section{Efron--Buchta identities}\label{sec:Efron-Buchta}

The Efron--Buchta identities \cites{Buchta:2005,Buchta:2023} carry over to the
$R$-ball-convex setting: this is stated in Theorem~\ref{thm:R-buchta}.  The combinatorial part of Buchta's argument applies
verbatim, the only geometric input that has to be replaced is Minkowski's
theorem $\conv A=\conv(\ext A)$, whose $R$-ball-convex counterpart for finite
sets is Lemma~\ref{lem:R-minkowski}.  What is needed there is not merely that a
single non-$R$-extremal point may be deleted, but that all of them may be
deleted simultaneously.

Despite the similarity in the arguments, we give a proof of Theorem~\ref{thm:R-buchta} for completeness.

\begin{proof}[Proof of Theorem~\ref{thm:R-buchta}]
    Recall that $0<r\leq R$ and $K\subset\RR^d$ is an $r$-ball-convex body.
    Since $K$ is $r$-ball-convex it is an intersection of balls of radius $r$,
    so $K\subset B^d(\bc,r)$ for some $\bc\in\RR^d$ and hence
    $\inte K\subset\inte B^d(\bc,r)$.  Almost surely the points
    $\bX_1,\dotsc,\bX_M$ lie in $\inte K$ and are pairwise distinct. Then
    $\|\bX_j-\bc\|<r\leq R$, so that $\bc\in\inte B^d(\bX_j,R)$ for every $j$ and
    therefore $\bc\in\inte C_R(\{\bX_1,\dotsc,\bX_M\})$.  Consequently,
    we have $C_R(\{\bX_1,\dotsc,\bX_M\})\neq \emptyset$ almost surely
    for all $M\geq 1$.

    Fix $n,m\geq1$ and let $M:=n+m$ and $A:=\{\bX_1,\dotsc,\bX_{n+m}\}$.  We
    claim that, almost surely,
    \begin{equation}\label{eqn:key-event}
        \begin{split}
            \bigl\{\bX_{n+1},\dotsc,\bX_{n+m}\bigr\}\subset [K]_n^R
            \iff \text{no $\bX_{n+\ell}$, $\ell\in[m]$, is $R$-extremal in $A$.}
        \end{split}
    \end{equation}
    Indeed, if $\bX_{n+\ell}\in[K]_n^R=\conv_R\{\bX_1,\dotsc,\bX_n\}$ for every
    $\ell$, then $\bX_{n+\ell}\in\conv_R(A\setminus\{\bX_{n+\ell}\})$ by the
    monotonicity of $\conv_R$, so no $\bX_{n+\ell}$ is $R$-extremal.
    Conversely, if no $\bX_{n+\ell}$ is $R$-extremal, then
    Lemma~\ref{lem:R-minkowski} applied with
    $S=\{\bX_{n+1},\dotsc,\bX_{n+m}\}$ gives
    $[K]_n^R=\conv_R(A\setminus S)=\conv_R(A)\supset S$.  This is the point at
    which the simultaneous deletion of all non-$R$-extremal points is used.

    Since $[K]_n^R\subset K$ by
    Lemma~\ref{lem:elementary}\,(c) and $\bX_1,\dotsc,\bX_n$ are independent, we have
    \begin{equation}\label{eqn:fubini}
        \frac{\EE\vol_d([K]_n^R)^m}{\vol_d(K)^m}
        = \PP\bigl(\bX_{n+1},\dotsc,\bX_{n+m}\in[K]_n^R\bigr).
    \end{equation}
    In order to prove \eqref{eqn:R-buchta-forward}, fix $n\geq2$ and $i\in\{2,\dotsc,n\}$ and
    apply \eqref{eqn:fubini} with $n$ replaced by $i$ and $m$ by $n-i$.  Its
    left-hand side is $\EE\vol_d([K]_i^R)^{n-i}/\vol_d(K)^{n-i}$, and by
    \eqref{eqn:key-event} its right-hand side is the probability that a fixed
    set of $n-i$ of the $n$ labels contains no point that is $R$-extremal in
    $\{\bX_1,\dotsc,\bX_n\}$.  Given $f_0([K]_n^R)=k$, the $k$ extremal points
    form a uniformly distributed $k$-element subset by exchangeability, so this
    conditional probability equals $\binom{n-k}{n-i}\big/\binom{n}{n-i}$.
    Since $\binom{n}{n-i}=\binom{n}{i}$, summing over $k$ yields \eqref{eqn:R-buchta-forward}.
    Note that only the terms with $2\leq k\leq i$
    contribute, because $\binom{n-k}{n-i}=0$ for $k>i$.

    Finally, \eqref{eqn:efron-buchta} follows from \eqref{eqn:R-buchta-forward} by
    the binomial inversion of \cite{Buchta:2005}*{Thm.\ 2}, using
    $\binom{n}{i}\binom{n-i}{n-k}=\binom{n}{k}\binom{k}{i}$. The summation
    starts at $i=2$ because $\vol_d([K]_1^R)=0$.
\end{proof}

We note that \eqref{eqn:R-buchta-forward} is equivalent to
\begin{equation*}%\label{eqn:R-buchta-moments}
    \frac{\EE\vol_d([K]_n^R)^m}{\vol_d(K)^m}
    = \EE \prod_{\ell=1}^m \Bigl(1- \frac{f_0([K]_{n+m}^R)}{n+\ell}\Bigr)
    \qquad\text{for all $n,m\geq1$}
\end{equation*}
(cf. the first statement of Theorem~1 and its proof in \cite{Buchta:2005}), which yields for $m=1$ the $R$-ball-convex version of Efron's identity, that is, under the same hypotheses as in Theorem~\ref{thm:R-buchta}, we have the following:
\begin{equation}\label{eqn:R-efron}
    \EE f_0([K]_{n+1}^R) = (n+1) \biggl( 1 - \frac{\EE\vol_d([K]_n^R)}{\vol_d(K)} \biggr).
\end{equation}
This was already derived for $d=2$ in \cite{FKV:2014}*{(5.10)}.

Moreover, for example, \eqref{eqn:efron-buchta} yields
\begin{align}
    \label{eqn:pn2}
    p_{n,2}^R(K) & = \binom{n}{2} \frac{\EE\vol_d([K]_{2}^R)^{n-2}}{\vol_d(K)^{n-2}},                     \\
    \label{eqn:pn3}
    p_{n,3}^R(K) & = \binom{n}{3} \frac{\EE\vol_d([K]_{3}^R)^{n-3}}{\vol_d(K)^{n-3}}-(n-2)\,p_{n,2}^R(K),
\end{align}
and, at the other end of the range,
\begin{equation}\label{eqn:pnn}
    p_{n,n}^R(K) = 1+\sum_{i=2}^{n-1} (-1)^{n-i} \binom{n}{i}
    \frac{\EE\vol_d([K]_{i}^R)^{n-i}}{\vol_d(K)^{n-i}}.
\end{equation}

We also note that due to the monotonicity of the $R$-ball-convex hull with respect to containment, if $r\leq R\leq R'$ then $f_0([K]_n^R)\leq f_0([K]_n^{R'})\leq f_0([K]_n)$ holds pointwise, and for
$k=2,\dotsc,n$,
\begin{equation*}
    \sum_{j=2}^{k} p_{n,j}^R(K) = \PP\bigl(f_0([K]_n^R)\leq k\bigr)
    \geq \PP\bigl(f_0([K]_n^{R'})\leq k\bigr) = \sum_{j=2}^{k} p_{n,j}^{R'}(K).
\end{equation*}
Moreover $f_0([K]_n^R)\to f_0([K]_n)$ almost surely as $R\to\infty$, so
that $\PP(f_0([K]_n^R)\leq k)$ decreases to $\PP(f_0([K]_n)\leq k)$ for
every $k$.  In particular, $p_{n,2}^R(K)$ decreases to $0$ and $p_{n,n}^R(K)$,
given by \eqref{eqn:pnn}, increases to $p_{n,n}(K)$ as $R\to\infty$.  However, the
argument does not give information on the individual probabilities
$p_{n,k}^R(K)$ with $2<k<n$, which are differences of two non-increasing
functions of $R$, see also Question~\ref{qu:monotone-p43}.

\begin{remark}[Efron-type identities for Poisson polytopes]%\label{rem:LM-poisson}
    Beermann and Reitzner \cite{BR:2015}*{Thm.~1} showed that, for Poisson
    polytopes, the Efron--Buchta identities can be strengthened to an identity between
    the generating function of the number of non-vertices and the moment
    generating function of the measure of the Poisson polytope.
    Very recently, Last and Molchanov \cite{LM:2026} obtained Efron-type
    identities for the hulls of Poisson processes generated by general
    stopping sets, see also \cite{LM:2025}, with the ball-convex hull as one
    of their examples \cite{LM:2026}*{Ex.~8.3}.  Their identities relate the
    number of Poisson points on the boundary of the hull to the intensity
    mass of the complement of the hull, and include joint distributions and
    variances.  In particular, \cite{LM:2026}*{Cor.~4.3} applied to the
    ball-convex hull of a Poisson process with uniform intensity on $K$ is
    the Poissonised form of \eqref{eqn:R-efron}, and
    \cite{LM:2026}*{Thm.~4.5} generalises the identity of Beermann and
    Reitzner.  However, as they remark \cite{LM:2026}*{Rem.~4.4}, their
    results do not seem to produce a Poissonised version of the higher-moment
    identities of Buchta \cites{Buchta:2005,Buchta:2023}.
\end{remark}

\section{Moments of the volume of a random ball-segment in a ball}\label{sec:two}

We now compute the moments of the volume of the $R$-ball-convex hull of two
independent, uniformly distributed random points in a Euclidean ball of radius
$r$, where $0<r\leq R$.  We first determine the volume of the standard
\emph{$R$-ball-segment}, or \emph{$R$-segment} for short,
$\mathrm{Sg}_R(\ell)=\conv_R\{-\ell\be_1,\ell\be_1\}\subset\RR^d$, the
$R$-ball-convex hull of two points at distance $2\ell\leq 2R$.
In the terminology of \cite{BLNP:2007} this is the $R$-spindle of the two
points $\pm\ell\be_1$.  For $d=2$ it may also be called a lens; for $d\geq 3$,
however, a lens commonly refers to the intersection of two balls, whereas the
$R$-segment is a body of revolution that arises as the intersection of infinitely
many $R$-balls.  For $d=3$ its boundary surface is classically known as
Kepler's lemon: rotating a circular arc about the chord joining its endpoints,
Kepler \cite{Kep1615} obtained an apple or a lemon according to whether the
arc is larger or smaller than a semicircle, and for $\ell\leq R$ the generating
arc of $\mathrm{Sg}_R(\ell)$ is at most a semicircle.

\subsection{Volume of a ball-segment}

Let $\kappa_{d}:=\vol_{d}(B^{d})$ and let $\ibeta(a,b;z)$ denote the incomplete beta function defined by
\begin{equation*}
    \ibeta(a,b;z) := \int_{0}^z t^{a-1} (1-t)^{b-1}\dint t, \quad \text{$a,b>0$ and $0\leq z \leq 1$},
\end{equation*}
and $\ibeta(a,b):=\ibeta(a,b;1)$.

\begin{lemma}\label{lem:segment-volume}
    Let $0<\ell\leq R$. Then
    \begin{align}
        \vol_d(\mathrm{Sg}_R(\ell))
        %& = R^d \vol_d\Bigl(\mathrm{Sg}_1\Bigl(\frac{\ell}{R}\Bigr)\Bigr)
        = 2\kappa_{d-1} \int_{0}^{\ell} \Bigl(\sqrt{R^2-s^2}-\sqrt{R^2-\ell^2}\Bigr)^{d-1}\, \dint s.
        \label{eqn:segment-vol}
    \end{align}
\end{lemma}
\begin{proof}
    By the homogeneity property in \eqref{eqn:equivariance} we find
    \begin{equation*}
        \vol_d(\mathrm{Sg}_R(\ell))
        = R^d \vol_d\biggl(\conv_1\Bigl\{\pm \frac{\ell}{R}\, \be_1\Bigr\}\biggr)
        = R^d \vol_d\biggl(\mathrm{Sg}_1\Bigl(\frac{\ell}{R}\Bigr)\biggr).
    \end{equation*}
    Therefore, it is sufficient to calculate $\vol_d(\mathrm{Sg}_1(\ell))$ for $0<\ell\leq 1$. Since the ball-segment is a body of revolution, we obtain by Fubini's theorem
    \begin{equation*}
        \vol_d(\mathrm{Sg}_1(\ell))
        = \int_{-\ell}^\ell \vol_{d-1}(\rho_{\ell}(s)B^{d-1}) \, \dint s
        = 2\kappa_{d-1} \int_{0}^\ell \rho_{\ell}(s)^{d-1}\,\dint s,
    \end{equation*}
    where
    \begin{equation*}
        \rho_{\ell}(s) := \sqrt{1-s^2}-\sqrt{1-\ell^2},
    \end{equation*}
    is the radius of the $(d-1)$-dimensional ball that is the intersection of $\mathrm{Sg}_1(\ell)$ with the hyperplane $\be_1^\bot+s\be_1$.
\end{proof}

For $\gamma := \arcsin\frac{\ell}{R} \in [0,\frac{\pi}{2}]$, Lemma \ref{lem:segment-volume} gives
\begin{equation}\label{eqn:segment-volume-Gd}
    \vol_d(\mathrm{Sg}_R(\ell))
    = 2\kappa_{d-1} R^d \int_{0}^{\sin\gamma} \Bigl(\sqrt{1-s^2}-\cos\gamma\Bigr)^{d-1}\,\dint s
    = R^d G_d(\gamma),
\end{equation}
where, using the substitution $s=\sin\varphi$, we define
\begin{equation*}%\label{eqn:GH}
    G_d(\gamma)  := 2\kappa_{d-1}\int_{0}^{\gamma}
    \bigl(\cos\varphi-\cos\gamma\bigr)^{d-1}\cos\varphi\,\dint\varphi
    = \vol_d\bigl(\mathrm{Sg}_1(\sin\gamma)\bigr).
\end{equation*}
In particular, for $d=2$, we find \eqref{eqn:segment-2d}.

We note that $R\mapsto V_R(\ell)$ is strictly decreasing and convex on $[\ell,\infty)$ \cite{AAF:2026}. Moreover,
\begin{align}\notag
    V_R(\ell) & = \pi\ell^2 - 2(2\ell)^{\frac 32} \sqrt{R-\ell} + O(R-\ell) \qquad \text{for $R\downarrow \ell$},
    \intertext{and}
    V_R(\ell) & = \frac{4\ell^3}{3R} +O(R^{-3}) \qquad \text{for $R\to \infty$}.
    \label{eqn:segment-2d-asymp}
\end{align}

\subsection{The random half-distance}

By Hammersley \cite{H50}, see also Lord \cites{L54a,L54b}, the \emph{random half distance}
\begin{equation*}
    L_{rB^d}:=\frac{1}{2}\|\bX-\bY\| \in [0,r]
\end{equation*}
of two independent uniform random points $\bX,\bY$ in $rB^d$ has density
\begin{align}\label{eqn:dist-density}
    f_{d,r}(\ell) = \frac{1}{r}f_{d,1}\Bigl(\frac{\ell}{r}\Bigr)
     & =\frac{d\,2^{d}\,\kappa_{d-1}}{\kappa_d}\, \ell^{d-1}r^{-d}\,
    \ibeta\biggl(\frac{d+1}{2}, \frac 12; 1-\frac{\ell^2}{r^2}\biggr),
\end{align}
for $\ell\in[0,r]$, and moments
\begin{equation}\label{eqn:L-moments}
    \EE L_{rB^d}^k = r^k \EE L_{B^d}^k
    = r^k \, \frac{d\,2^d\, \kappa_{d-1}}{(d+k)\,\kappa_d}\, \ibeta\Bigl(\frac{d+k+1}{2}, \frac{d+1}{2}\Bigr),
\end{equation}
see Hammersley \cite{H50}*{eqn.~(21)}.

\begin{remark}[a beta factorization]%\label{rem:beta-product}
    The law of $L_{B^d}$ belongs to the beta-product family. Applying the
    Legendre duplication formula to \eqref{eqn:L-moments} gives
    \begin{equation*}
        \EE L_{B^d}^{2s}
        = \frac{\bigl(\frac d2\bigr)_s\,\bigl(\frac{d+1}{2}\bigr)_s}
        {\bigl(\frac d2+1\bigr)_s\,(d+1)_s},
        \qquad s>0,
    \end{equation*}
    where $(a)_s=\Gamma(a+s)/\Gamma(a)$, and hence, by moment determinacy,
    $L_{B^d}^2$ has the same distribution as $B_1B_2$ for independent
    $B_1\sim\operatorname{Beta}\bigl(\frac d2,1\bigr)$ and
    $B_2\sim\operatorname{Beta}\bigl(\frac{d+1}{2},\frac{d+1}{2}\bigr)$.
    Probabilistically, this reflects Kingman's random secant decomposition
    \cite{K69}: conditionally on the line spanned by $\bX$ and $\bY$, the
    half-distance factors as $L_{B^d}=hT$ with
    $T\sim\operatorname{Beta}(d,2)$ independent of the half-chord length $h$
    of the secant, whose square satisfies
    $h^2\sim\operatorname{Beta}\bigl(\frac{d+3}{2},\frac{d-1}{2}\bigr)$.
\end{remark}

Note that the substitution $t=\sin^2\varphi$ transforms the incomplete beta function in \eqref{eqn:dist-density} to
\begin{equation*}
    \ibeta\biggl(\frac{d+1}{2}, \frac 12; 1-\frac{\ell^2}{r^2}\biggr)
    = \int_{0}^{1-\frac{\ell^2}{r^2}} t^{\frac{d-1}{2}} (1-t)^{-\frac{1}{2}} \, \dint t
    = 2 \int_{0}^{\arccos\frac{\ell}{r}} \sin^d \varphi\,\dint\varphi.
\end{equation*}
Thus, for $\alpha:=\arccos\frac{\ell}{r}$, we find
\begin{equation}\label{eqn:density-segment-trig}
    f_{d,r}(\ell) = \frac{d\,2^{d+1}\,\kappa_{d-1}}{\kappa_d}\, r^{-1}\,
    (\cos \alpha(\ell))^{d-1}\, H_d(\alpha(\ell)),
\end{equation}
where we define
\begin{equation}\label{eqn:def-Hd}
    H_d(\alpha) := \int_{0}^\alpha \sin^d\varphi\,\dint\varphi
    \qquad \text{for $\alpha\in [0,\tfrac{\pi}{2}]$.}
\end{equation}

\subsection{Moments of the volume of a random ball-segment}

We are interested in the moments of the volume functional of the $R$-segment generated by two random points in $rB^d$, that is,
\begin{equation}\label{eqn:segment-scaling}
    \EE\vol_d\bigl([rB^d]_2^R\bigr)^k = \EE\vol_d(\mathrm{Sg}_R(L_{rB^d}))^k
    = R^{dk} \; \EE\vol_d(\mathrm{Sg}_{1}(\tfrac{r}{R}L_{B^d}))^k.
\end{equation}
We give a closed form for these moments by using Hammersley's density \eqref{eqn:dist-density}.

\begin{proposition}\label{prop:k-segment-moments}
    Let $d\geq 2$, $k\in\NN_0$ and $0<r\leq R$. Then
    \begin{align}
        \EE\vol_d\Bigl(\mathrm{Sg}_R\bigl(L_{rB^d}\bigr)\Bigr)^{\!k}
         & = \frac{2^{d+1}\, d\,\kappa_{d-1}}{\kappa_d}\; R^{dk}
        \int_{0}^{\frac{\pi}{2}} G_d\bigl(\gamma(\alpha)\bigr)^{\!k}\, H_d(\alpha)\,
        \cos^{d-1}\!\alpha\;\sin \alpha \, \dint \alpha,
        \label{eqn:k-vol_d-segment-trig}
    \end{align}
    where $\gamma(\alpha):=\arcsin\bigl(\eta\cos \alpha\bigr)$ with $\eta:=\frac{r}{R}\in(0,1]$.
\end{proposition}

\begin{proof}
    By homogeneity \eqref{eqn:equivariance} and \eqref{eqn:segment-volume-Gd}
    \begin{align*}
        \EE\vol_d\Bigl(\mathrm{Sg}_R\bigl(L_{rB^d}\bigr)\Bigr)^{\!k}
         & = R^{dk}\,\EE\vol_d\Bigl(\mathrm{Sg}_1\bigl(\tfrac{r}{R}L_{B^d}\bigr)\Bigr)^{\!k}  \notag      \\
         & = R^{dk}\int_{0}^{1} G_d\Bigl(\arcsin \frac{r\ell}{R}\Bigr)^{\!k} f_{d,1}(\ell)\, \dint \ell .
    \end{align*}
    With the substitutions $\sin \gamma = \frac{r\ell}{R}$ and $\cos\alpha=\ell$, and applying
    \eqref{eqn:density-segment-trig}, we find \eqref{eqn:k-vol_d-segment-trig}.
\end{proof}

The integral in \eqref{eqn:def-Hd} is elementary in every fixed dimension.  For $d=2$, \eqref{eqn:density-segment-trig} yields
\begin{align}
    f_{2,r}(\ell)
     & =\frac{16}{\pi} \ell r^{-2}\Biggl (\arccos\Bigl(\frac{\ell}{r}\Bigr)
    -\frac{\ell}{r}\sqrt{1-\frac{\ell^2}{r^2}}\Biggr ),\label{eqn:dist-density-2d}
\end{align}
and \eqref{eqn:k-vol_d-segment-trig} gives
\begin{align}
    \EE\vol_2([B^2]_2^R)^k
     & =\EE V_R(L_{B^2})^k
    =\frac{4}{\pi}R^{2k}\int_0^{\frac{\pi}{2}} \bigl (2\gamma-\sin2\gamma\bigr )^{\!k}
    \bigl(2\alpha-\sin 2\alpha\bigr)\,(\sin 2\alpha)\, \dint \alpha,
    \label{eqn:k-moment}
\end{align}
for all $R\geq 1$, where $\cos \alpha= \ell$ and  $\sin \gamma = \frac{\ell}{R}$.

In particular, \eqref{eqn:k-moment} with $k=1$ and $k=2$ evaluates the first two
of the three integrals
\begin{equation}\label{eqn:m1m2}
    m_1(R) := \EE\vol_2\bigl([B^2]_2^R\bigr)
    \qquad\text{and}\qquad
    m_2(R) := \EE\vol_2\bigl([B^2]_2^R\bigr)^2
\end{equation}
that carry the whole dependence on $R$ in
Theorem~\ref{thm:sylvester-r-convex}; the third, the mixed moment
$\mu(R)=\EE[\Delta_R(L_{B^2})V_R(L_{B^2})]$, is calculated in
Section~\ref{sec:three}.

At the smallest admissible radius $R=1$ the integral \eqref{eqn:k-moment}
degenerates and can be evaluated in closed form. Indeed, $\sin\gamma=\cos \alpha$
forces $\gamma=\frac{\pi}{2}-\alpha$, so that $2\gamma-\sin2\gamma=\pi-2\alpha-\sin2\alpha$,
and the substitution $u=2\alpha$ turns \eqref{eqn:k-moment} into
\begin{equation}\label{eqn:k-moment-at-one}
    \EE\vol_2\bigl([B^2]_2^1\bigr)^k
    = \frac{2}{\pi}\int_{0}^{\pi}\bigl(\pi-u-\sin u\bigr)^{k}
    \bigl(u-\sin u\bigr)\sin u\,\dint u ,
\end{equation}
whose integrand is a polynomial in $u$, $\sin u$ and $\cos u$.  Evaluating
\eqref{eqn:k-moment-at-one} for $k=1$ and $k=2$ yields
\begin{equation}\label{eqn:m1m2-at-one}
    m_1(1) = \frac{32}{3\pi}-\pi \approx 0.25372
    \qquad\text{and}\qquad
    m_2(1) = \frac{81-8\pi^2}{12} \approx 0.17026 .
\end{equation}

Notice also that for every fixed $d$ and $k$, the integrand of \eqref{eqn:k-vol_d-segment-trig}
is elementary: $G_d$ and $H_d$ are, by expansion of the binomial
$(\cos\varphi-\cos\gamma)^{d-1}$, finite combinations of the elementary
integrals $\int_0^\gamma\cos^j\varphi\,\dint\varphi$ and
$\int_0^a\sin^d\varphi\,\dint\varphi$.

\begin{corollary}\label{cor:pn2}
    Under the hypotheses of Proposition~\ref{prop:k-segment-moments},
    we find, by \eqref{eqn:pn2} and
    \eqref{eqn:segment-scaling},
    \begin{align*}
        p_{n,2}^R(rB^d)  =
        \binom{n}{2} \frac{\EE\vol_d([rB^d]_2^R)^{n-2}}{\vol_d(rB^d)^{n-2}}
        =p_{n,2}^1\Bigl(\frac{r}{R}B^d\Bigr) \leq p_{n,2}^1(B^d),
    \end{align*}
    the inequality because $p_{n,2}^1(\eta B^d)=p_{n,2}^{1/\eta}(B^d)$
    by \eqref{eqn:equivariance}, with $1/\eta\geq1$, and because
    $R\mapsto p_{n,2}^R$ is non-increasing, see the discussion following \eqref{eqn:pnn}.
    In particular, for the two-dimensional disc
    $B^2$ and $R\geq 1$ we obtain, in the notation \eqref{eqn:m1m2},
    \begin{align*}
        p_{3,2}^R(B^2) = \frac{3}{\pi}\, m_1(R)   & \leq p_{3,2}^1(B^2) = \frac{32}{\pi^2} - 3 \approx 24.2278\%, \\
        \intertext{and}
        p_{4,2}^R(B^2) = \frac{6}{\pi^2}\, m_2(R) & \leq p_{4,2}^1(B^2) = \frac{81}{2\pi^2}-4 \approx 10.3508\%.
    \end{align*}
\end{corollary}

The moment that will be needed in Section~\ref{sec:three} is the second one,
$k=2$, in the planar case $d=2$, $r=1$.  By \eqref{eqn:k-moment} we have, with
$\alpha=\frac\pi 2-\theta$,
\begin{align*}
    m_2(R) & = \EE \vol_2\bigl([B^2]_2^R\bigr)^2 = \EE V_R(L_{B^2})^2.
\end{align*}
By \eqref{eqn:dist-density-2d} and inserting the expansion \eqref{eqn:segment-2d-asymp} with $\ell=\sin\theta$, we obtain
\begin{equation}\label{eqn:m2-asymptotics}
    m_2(R)  = \frac{A}{R^2} \bigl(1+O(R^{-2})\bigr),
\end{equation}
for $R\to \infty$, where
\begin{align*}
    A & = \frac{64}{9\pi} \int_{0}^{\frac{\pi}{2}} (\sin\theta)^6 (\pi-2\theta-\sin 2\theta)(\sin 2\theta)\, \dint\theta
    =\frac{7}{72}.
\end{align*}

\section{The \texorpdfstring{$R$}{R}-ball-convex hull of three points}
\label{sec:three}

For $0\leq\ell\leq R$ we abbreviate
\begin{equation*}%\label{eqn:gamma-h}
    \begin{split}
         & h := \sqrt{R^2-\ell^2},
        \qquad
        \gamma := \arcsin\frac{\ell}{R}\in\Bigl[0,\tfrac{\pi}{2}\Bigr], \\
         & \qquad\text{so that}\qquad
        \sin\gamma=\frac{\ell}{R},\quad \cos\gamma=\frac{h}{R} .
    \end{split}
\end{equation*}
Recall from \eqref{eqn:segment-2d} and \eqref{eqn:def-Delta-R} that
$V_R(\ell) = \vol_2\bigl(\mathrm{Sg}_R(\ell)\bigr) = 2R^2\gamma - 2\Delta_R(\ell)$,
and $\Delta_R(\ell)=\ell h$.
Differentiation gives
\begin{equation}\label{eqn:V-derivative}
    V_R'(\ell) = \frac{4\ell^2}{\sqrt{R^2-\ell^2}}\,,
    \qquad 0\leq\ell<R,
\end{equation}
so $V_R$ is non-negative, strictly increasing, and convex on $[0,R]$.
Note that for $d\geq 2$ \cite{AAF:2026}*{Prop.\ 3.2} determines the convexity of the map $\bx\in\RR^d\mapsto \vol_d(\conv_R\{\bx,\bx_0\})$ for fixed $\bx_0\in\RR^d$,
which for $d=2$ also implies that $V_R(\ell) = \vol_2(\conv_R\{-\ell\be_1,\ell\be_1\})$ is non-decreasing and convex.

\subsection{The combinatorial structure of \texorpdfstring{$\conv_R$}{conv-R} of three points}

The whole calculation is organised by the following elementary but
decisive observation: the circumradius alone decides everything.  The two cases
are drawn in Figure~\ref{fig:structure}.

\begin{lemma}\label{lem:structure}
    Let $\bx_1,\bx_2,\bx_3\in\RR^2$ be affinely independent points that are
    contained in some closed disc of radius $R>0$.  Let
    $T:=\conv\{\bx_1,\bx_2,\bx_3\}$, let $\varrho$ be the circumradius of $T$,
    let $\ell_{ij}:=\tfrac12\|\bx_i-\bx_j\|$, and let $\theta_j$ denote the angle of
    $T$ at $\bx_j$.  Then, for $\{i,j,k\}=\{1,2,3\}$ the following hold.
    \begin{sloppypar}
        \begin{enumerate}
            \item[(i)] $\bx_j \in \conv_R\{\bx_i,\bx_k\}$ if and only if
                  $\theta_j\geq \pi/2$ and $\varrho\geq R$.
            \item[(ii)] If $\varrho<R$, then $f_0(\conv_R\{\bx_1,\bx_2,\bx_3\})=3$,
                  the boundary of $\conv_R\{\bx_1,\bx_2,\bx_3\}$ is the union of
                  three arcs of radius $R$ meeting pairwise at
                  $\bx_1,\bx_2,\bx_3$ at interior angles $<\pi$, and
                  \begin{equation}\label{eqn:rtriangle}
                      \vol_2\bigl(\conv_R\{\bx_1,\bx_2,\bx_3\}\bigr)
                      = \vol_2(T) + \frac{1}{2}\sum_{1\leq i<j\leq 3} V_R(\ell_{ij}).
                  \end{equation}
            \item[(iii)] If $\varrho>R$, then $T$ has an obtuse angle at, say, $\bx_j$ and
                  \begin{equation*}
                      \conv_R\{\bx_1,\bx_2,\bx_3\} = \conv_R\{\bx_i,\bx_k\}
                      = \mathrm{Sg}_R\bigl(\ell_{ik}\bigr),
                  \end{equation*}
                  where $\ell_{ik}=\max\{\ell_{12},\ell_{13},\ell_{23}\}$.  In particular
                  $f_0(\conv_R\{\bx_1,\bx_2,\bx_3\})=2$ and
                  $\vol_2(\conv_R\{\bx_1,\bx_2,\bx_3\}) = V_R(\ell_{ik})$.
        \end{enumerate}
    \end{sloppypar}
    Since $\varrho$ does not depend on $j$, and since a triangle has at most one
    angle $\geq\pi/2$, the three events that $\bx_j \in \conv_R\{\bx_i,\bx_k\}$ in \textup{(i)} are pairwise disjoint.
\end{lemma}

\begin{figure}[t]
    \centering
    \resizebox{\linewidth}{!}{%
        \begin{tikzpicture}[scale=2.75, line cap=round, line join=round]

            %% ---- Panel (a): rho < R ----
            \begin{scope}
                \coordinate (x1) at (-0.2084, 1.1818);
                \coordinate (x2) at (-0.9830,-0.6883);
                \coordinate (x3) at ( 1.1276,-0.4104);
                \coordinate (O)  at (0,0);

                \fill[blue!15]
                (x1) arc[start angle=127.100, end angle=187.900, radius=2]
                arc[start angle=-114.655, end angle=-50.345, radius=2]
                arc[start angle=8.694,   end angle=71.306,  radius=2]
                -- cycle;
                \fill[blue!5] (x1) -- (x2) -- (x3) -- cycle;

                \draw[dashed, gray] (O) circle[radius=1.2];
                \draw[gray, shorten >=1pt] (O) -- ($(O)+(120:1.2)$) node[midway, right=-1pt, black] {$\varrho$};
                \fill[gray] (O) circle[radius=0.7pt];

                \draw[thick] (x1) arc[start angle=127.100, end angle=187.900, radius=2];
                \draw[thick] (x2) arc[start angle=-114.655, end angle=-50.345, radius=2];
                \draw[thick] (x3) arc[start angle=8.694,   end angle=71.306,  radius=2];

                \draw (x1) -- (x2);
                \draw (x2) -- (x3) node[midway, below=-1pt] {$2\ell_{23}$};
                \draw (x3) -- (x1);

                \draw ($(x2)+(7.5:0.25)$) arc[start angle=7.5, end angle=67.5, radius=0.25];
                \node at ($(x2)+(37.5:0.42)$) {$\theta_2$};

                \fill (x1) circle[radius=1.pt] node[above=1pt] {$\bx_1$};
                \fill (x2) circle[radius=1.pt] node[below left=-1pt] {$\bx_2$};
                \fill (x3) circle[radius=1.pt] node[below right=-2pt] {$\bx_3$};

                % segment label, next to the bulge over l_13 (no arrow)
                \node[font=\small, rotate=-50] at (0.572,0.479) {$\tfrac12 V_R(\ell_{13})$};
                % hull label, inside the triangle
                \node[font=\small] at (0.05,-0.35) {$\conv_R\{\bx_1,\bx_2,\bx_3\}$};

                \node at (0,-1.62) {(a)\; $\varrho<R$:\ $f_0=3$};
            \end{scope}

            %% ---- Panel (b): rho > R ----
            \begin{scope}[xshift=3.55cm, yshift=0.35cm]
                \coordinate (y1) at (-0.9,0);
                \coordinate (y2) at ( 0.9,0);
                \coordinate (y3) at ( 0.15,0.33);
                \coordinate (cm) at (0,-0.4359);
                \coordinate (O2) at (0,-1.0282);

                \fill[blue!15]
                (y1) arc[start angle=154.158, end angle=25.842,   radius=1]
                arc[start angle=-25.842, end angle=-154.158, radius=1]
                -- cycle;
                \draw[dotted, gray] (cm) circle[radius=1];

                \fill[white,opacity=0.5] (y1) -- (y2) -- (y3) -- cycle;

                \draw[gray, shorten >=1pt] (cm) -- ($(cm)+(240:1)$) node[pos=1, below left=-3pt, black] {$R$};
                \fill[gray] (cm) circle[radius=0.8pt];

                \draw[thick] (y1) arc[start angle=154.158, end angle=25.842,   radius=1];
                \draw[thick] (y2) arc[start angle=-25.842, end angle=-154.158, radius=1];

                % upper part of the circumcircle only
                \draw[dashed, gray] ($(O2)+(148:1.3664)$) arc[start angle=148, end angle=32, radius=1.3664];
                \draw[gray, shorten >=1pt] (O2) -- ($(O2)+(35:1.3664)$) node[midway, below right=-2.5pt, black] {$\varrho$};
                \fill[gray] (O2) circle[radius=0.8pt];

                \draw (y1) -- (y3) -- (y2);
                \draw (y1) -- (y2) node[pos=0.72, below=0pt] {$2\ell_{12}$};
                \draw ($(y3)+(-162.56:0.15)$) arc[start angle=-162.56, end angle=-23.75, radius=0.15];
                \node[font=\small] at ($(y3)+(-95:0.25)$) {$\theta_3$};

                \fill (y1) circle[radius=1.pt] node[left=1pt] {$\bx_1$};
                \fill (y2) circle[radius=1.pt] node[right=1pt] {$\bx_2$};
                \fill (y3) circle[radius=1.pt] node[above=1.5pt] {$\bx_3$};

                % ball-segment label, above the ball-segment (no arrow)
                \node[font=\small] at (-0.30,-0.20) {$\mathrm{Sg}_R\!\bigl(\ell_{12}\bigr)$};

                \node at (0,-1.97) {(b)\; $\varrho>R$:\ $f_0=2$};
            \end{scope}

        \end{tikzpicture}%
    }
    \caption{The two cases of Lemma~\ref{lem:structure}, decided by the
        circumradius $\varrho$.
    }
    \label{fig:structure}
\end{figure}
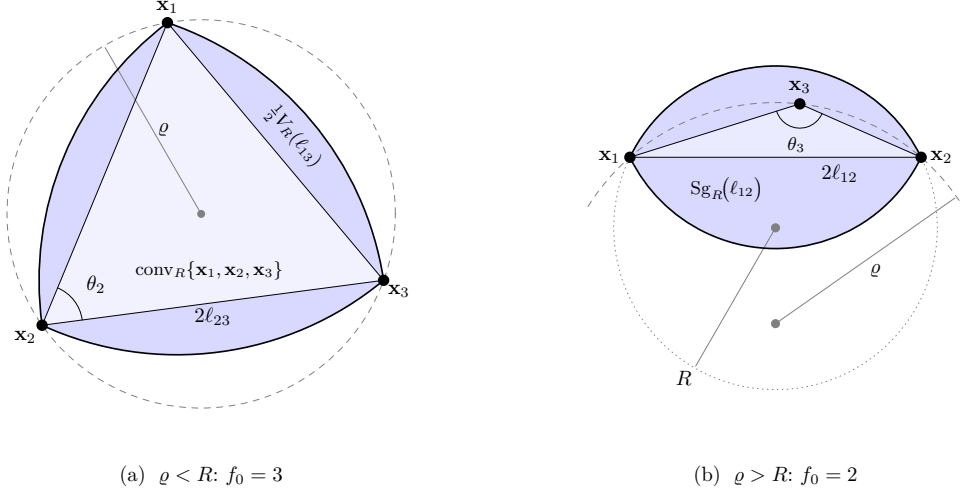

\begin{proof}
    (i) Put $\ell:=\ell_{ik}$ and assume, after a rigid motion, that
    $\bx_i=-\ell\be_1$, $\bx_k=\ell\be_1$ and that $\bx_j$ lies in the upper
    half-plane.  Since the three points are in a common $R$-disc, we have
    $\ell\leq R$, and
    \begin{equation*}
        \mathrm{Sg}_R(\ell) = B^2(\bc_-,R)\cap B^2(\bc_+,R),
        \qquad\text{with } \bc_\pm := \pm\sqrt{R^2-\ell^2}\,\be_2 .
    \end{equation*}
    Since $\bx_j$ is in the upper half-plane, it is clear that $\bx_j\in\mathrm{Sg}_R(\ell)$ if and only if $\bx_j\in B^2(\bc_-,R)$.
    By the
    inscribed angle theorem, for $\beta\in(0,\pi)$ the locus
    $\{\bp=(p_1, p_2)\colon p_2>0,\ \angle\,\bx_i\bp\,\bx_k=\beta\}$ is a circular arc through
    $\bx_i$ and $\bx_k$. These arcs foliate the open upper half-plane, and the
    circular segments they determine are nested, and decreasing in $\beta$.  For
    $\beta=\pi-\gamma$ the arc is $\partial B^2(\bc_-,R)\cap\{p_2>0\}$ and the
    corresponding segment is $B^2(\bc_-,R)\cap\{p_2>0\}$.  Hence,
    \begin{equation*}
        \bx_j\in B^2(\bc_-,R)\iff \theta_j\geq \pi-\gamma
        \iff \Big(\theta_j\geq \tfrac{\pi}{2}\ \text{ and }\ \sin\theta_j\leq \sin\gamma = \tfrac{\ell}{R}\Big).
    \end{equation*}
    (For the second equivalence note that $\theta_j\geq\pi-\gamma$ forces
    $\theta_j\geq\pi/2$, and that on $[\pi/2,\pi)$ the condition
    $\theta_j\geq\pi-\gamma$ reads $\pi-\theta_j\leq\gamma$, both sides lying in
    $[0,\pi/2]$.)  By the law of sines $\varrho=\ell/\sin\theta_j$, so
    $\sin\theta_j\leq \ell/R$ is equivalent to $\varrho\geq R$.  This proves (i).

    (ii) Let $\varrho<R$.  By (i) no $\bx_j$ lies in the $R$-ball-segment of the
    other two, so all three points are $R$-extremal and $f_0=3$.  In particular
    $A:=\{\bx_1,\bx_2,\bx_3\}$ lies in an open $R$-disc, i.e.\
    $\inte C_R(A)\neq\emptyset$, so \eqref{lem:ball-polygon} applies and
    $C_R(A)$ has exactly $f_0=3$ vertices $\bc_{12},\bc_{13},\bc_{23}$, labelled
    so that $\|\bc_{ij}-\bx_i\|=\|\bc_{ij}-\bx_j\|=R$. Here
    $\|\bc_{ij}-\bx_k\|<R$, since equality would make $\bc_{ij}$ the
    circumcentre of $T$ and force $\varrho=R$.
    Writing $D_{ij}:=B^2(\bc_{ij},R)$ for the unique $R$-disc through $\bx_i$ and
    $\bx_j$ that contains $\bx_k$ in its interior, \eqref{lem:ball-polygon}
    gives
    \begin{equation*}
        \conv_R\{\bx_1,\bx_2,\bx_3\} = D_{12}\cap D_{13}\cap D_{23}.
    \end{equation*}
    The centre $\bc_{ij}$ lies on the same side of the line through
    $\bx_i,\bx_j$ as $\bx_k$: placing $\bx_i=-\ell\be_1$, $\bx_j=\ell\be_1$ with
    $\ell=\ell_{ij}$ and $\bx_k=(a,b)$, $b>0$, the disc of radius $R$ with centre $(0,-h)$ cannot contain $\bx_k$ as $\bx_k\notin \mathrm{Sg}_R(\ell)$.
    Consequently, $D_{ij}$ cuts off, on the side of
    $\mathrm{aff} \{\bx_i,\bx_j\}$ not containing $\bx_k$, the \emph{minor} circular segment,
    of area $\tfrac12V_R(\ell_{ij})$ because $\mathrm{Sg}_R(\ell_{ij})$ is the
    union of two congruent such segments.  As $D_{12}\cap D_{13}\cap D_{23}$ is
    the union of $T$ and of these three pairwise non-overlapping segments, one
    over each side of $T$, \eqref{eqn:rtriangle} follows; the boundary
    accordingly consists of the three arcs $\partial D_{ij}\cap\partial(D_{12}\cap
        D_{13}\cap D_{23})$, which meet at $\bx_1,\bx_2,\bx_3$ at interior angles
    strictly between the corresponding angle $\theta_j$ of $T$ and $\pi$.

    (iii) If $T$ were non-obtuse, then the smallest enclosing disc of
    $\{\bx_1,\bx_2,\bx_3\}$ would be the circumdisc, whence $\varrho\leq R$ by
    hypothesis. So $T$ is obtuse at, say, $\bx_j$, and $\ell_{ik}$ is the longest
    side.  By (i), $\bx_j\in\conv_R\{\bx_i,\bx_k\}$, so $\bx_j$ is not
    $R$-extremal and $\conv_R\{\bx_1,\bx_2,\bx_3\}=\conv_R\{\bx_i,\bx_k\}$.
\end{proof}

The excluded case $\varrho=R$ occurs with probability $0$.

\subsection{The distribution of the circumradius}%\label{sec:circumradius-d}

Lemma~\ref{lem:structure} identifies the event $\{f_0([K]_3^R)=2\}$ with a
purely metric event, and thereby identifies Sylvester's three-point problem for
the $R$-ball-convex hull with the distribution of the circumradius.
We denote by $\varrho([B^d]_n)$ the random circumradius of $n$ independent uniform random points $\bX_1,\dotsc,\bX_n$ in $B^d$ taken within the almost surely $(n-1)$-dimensional affine subspace they span.

\begin{theorem}[circumradius of random points]\label{thm:circumradius-d}
    For every $3\leq n \leq d+1$ and $R\geq1$,
    \begin{equation}\label{eqn:circumradius-n}
        \PP(\varrho([B^d]_n)>R) \geq 1-p_{n,n}^R(B^d).
    \end{equation}
    Moreover, for $n=3$ we have
    \begin{equation}\label{eqn:circumradius-d}
        \PP(\varrho([B^d]_3)>R) = p_{3,2}^R(B^d) = \frac{3}{\kappa_d}\,\EE \vol_d\!\bigl([B^d]_2^R\bigr).
    \end{equation}
    As $R\to\infty$,
    \begin{equation}\label{eqn:circumradius-d-tail}
        \begin{split}
            \PP(\varrho([B^d]_3)>R) & = \frac{C_d}{R^{d-1}}\Bigl(1+O\!\bigl(R^{-2}\bigr)\Bigr), \\
            C_d                     & := \frac{6d}{3d-1}\,\frac{\kappa_{d-1}^2}{\kappa_d^2}\,
            \ibeta\bigl(d,\tfrac12\bigr)\, \ibeta\Bigl(\tfrac{3d}{2},\tfrac{d+1}{2}\Bigr).
        \end{split}
    \end{equation}
    In particular $C_2=\frac{1024}{525\pi^2}$ and $C_3=\frac{3}{55}$.
\end{theorem}

\begin{proof}
    By Proposition~\ref{prop:pole}, if $\varrho([B^d]_n)<R$, then $f_0([B^d]_n)=n$. Thus
    \begin{equation*}
        \PP\Bigl(\varrho([B^d]_n)<R\Bigr) \leq p_{n,n}^R(B^d),
    \end{equation*}
    which implies \eqref{eqn:circumradius-n} since the event $\{\varrho([B^d]_n) = R\}$ occurs with probability $0$.
    For $n=3$ we may apply Lemma~\ref{lem:structure} in the affine plane $E$ that is almost surely spanned by $[B^d]_3$ to find that
    $f_0([B^d]_3)=3$ implies $\varrho([B^d]_3)\leq R$. Thus we conclude \eqref{eqn:circumradius-d}.

    For the tail asymptotics we first note
    $\sqrt{R^2-s^2}-\sqrt{R^2-\ell^2}
        =(\ell^2-s^2)\bigl(\sqrt{R^2-s^2}+\sqrt{R^2-\ell^2}\bigr)^{-1}$ with
    $\sqrt{R^2-\ell^2}\leq\sqrt{R^2-s^2}\leq R$, and
    $\int_{-\ell}^{\ell}(\ell^2-s^2)^{d-1}\dint s=\ibeta\bigl(\tfrac12,d\bigr)
        \ell^{\,2d-1}$. Thus, by \eqref{eqn:segment-vol}, we obtain uniformly in $0<\ell\leq1$ and $R\geq\sqrt2$
    \begin{equation}\label{eqn:seg-asymp}
        \vol_d\bigl(\mathrm{Sg}_R(\ell)\bigr)
        = \frac{\kappa_{d-1}}{(2R)^{d-1}}\,
        \ibeta\bigl(d,\tfrac12\bigr)\ell^{\,2d-1}
        \Bigl(1+O\!\bigl(R^{-2}\bigr)\Bigr).
    \end{equation}
    Combining \eqref{eqn:circumradius-d} and \eqref{eqn:seg-asymp} gives
    \eqref{eqn:circumradius-d-tail} with
    $C_d=\frac{3\kappa_{d-1}}{\kappa_d}\,2^{1-d}\,
        \ibeta\bigl(d,\tfrac12\bigr)\EE L_{B^d}^{2d-1}$, which is the stated constant, since by \eqref{eqn:L-moments}
    \begin{equation*}%\label{eqn:L-moment}
        \EE L_{B^d}^{2d-1}
        = \frac{2^d\,d\,\kappa_{d-1}}{(3d-1)\,\kappa_d}\,
        \ibeta\Bigl(\frac{3d}{2},\frac{d+1}{2}\Bigr). \qedhere
    \end{equation*}
\end{proof}

\begin{remark}
    One can show that inequality \eqref{eqn:circumradius-n} is strict for $n\geq 4$.
\end{remark}

At the critical radius $R=1$ the right-hand side of \eqref{eqn:circumradius-d} is
elementary in every dimension: there $\eta=1$ and
$\gamma(\alpha)=\arcsin(\cos \alpha)=\frac\pi2-\alpha$, so the integrand of
\eqref{eqn:k-vol_d-segment-trig} becomes a polynomial in $\alpha$, $\sin \alpha$ and
$\cos \alpha$. For example,
\begin{align*}
    \PP(\varrho([B^2]_3)>1) & =p^1_{3,2}(B^2)= \frac{32}{\pi^2}-3\approx24.23\%,                  \\
    \PP(\varrho([B^3]_3)>1) & =p^1_{3,2}(B^3)=\frac{20509}{1225}-\frac{27\pi^2}{16}\approx8.71\%.
\end{align*}

The tail estimate \eqref{eqn:circumradius-d-tail} implies that
$\EE\varrho([B^d]_3)^{\,k}<\infty$ if and only if $k<d-1$.
This follows by the layer-cake formula: first, let $k<d-1$. By \eqref{eqn:circumradius-d-tail} there exists $R_0>1$
such that the multiplicative error term $(1+O(1/R^2))$ is
bounded above by $2$. Then
\begin{equation*}
    \EE \varrho([B^d]_3)^k
    = \int_0^{\infty} k R^{k-1} \PP(\varrho([B^d]_3)>R) \, \dint R
    \leq R_0^k + 2k C_d \int_{R_0}^{\infty}  R^{k-d}\, \dint R
    < \infty.
\end{equation*}
Similarly, if $k\geq d-1$ then there exists $R_0>1$ such that the error in \eqref{eqn:circumradius-d-tail} is bounded below by $\frac{1}{2}$ and we find $\EE \varrho([B^d]_3)^k  = \infty$.
Note, in particular, that the mean circumradius is infinite for $d=2$ and finite for every $d\geq3$.

\begin{remark}[comparison with the literature]%\label{rem:circumradius-literature}
    The circumradius of a uniform sample in a ball has been studied before,
    conditionally on a containment event.  Affentranger
    \cite{Affentranger:1989} showed that the circumdisc of three independent
    uniform random points in $B^2$ is contained in $B^2$ with probability
    $2/5$ (more generally, he treated circumcircles of three random points
    in $B^d$), and Le~Ca\"er \cite{LeCaer:2017}*{Secs.~IV--V} determined, for
    every $d\geq2$, the joint distribution of the circumradius and of the
    distance of the circumcentre from the orthogonal projection of the origin
    onto the plane of the three points, conditioned on this containment.
    In particular, for $d=2$ the conditional density of $\varrho([B^2]_3)$
    is $60\varrho^3(1-\varrho)^2$ on $(0,1)$, that is, it has a $\mathrm{Beta}(4,3)$ distribution.
    Containment is the event $\{\varrho([B^d]_3)+\|\bc([B^d]_3)\|\leq1\}$, where $\bc([B^d]_3)$
    denotes the circumcentre of $[B^d]_3$. It forces
    $\varrho([B^d]_3)\leq1$, with equality only if $\bc([B^d]_3)$ is the origin, whereas
    \eqref{eqn:circumradius-d} describes $\varrho([B^d]_3)$ on $[1,\infty)$.
    The two results are thus supported on complementary ranges and do not overlap.
    Consistently, the containment probability $2/5$ does not exceed
    $\PP(\varrho([B^2]_3)\leq1)=4-\frac{32}{\pi^2}\approx75.77\%$.
    The unconditional distribution seems to have remained open: its tail was
    investigated in \cite{LeCaer:2017}*{Sec.~VIII} by Monte Carlo simulation,
    leading to the hypothesis that the density of $\varrho([B^2]_3)$ decays like a
    multiple of $\varrho^{-2}$. A fit of the sample maxima to a Fr\'echet
    distribution gave shape parameter $0.997\pm0.008$ and, for the maximum of
    $k$ samples, scale $(0.197\pm0.001)\,k$.
    Theorem~\ref{thm:circumradius-d} confirms this hypothesis for the disc
    and identifies the constant: $\PP(\varrho([B^2]_3)>R)=\frac{1024}{525\pi^2R}
        +O\!\bigl(R^{-3}\bigr)$, and $\frac{1024}{525\pi^2}=0.19762\ldots$ lies
    within the confidence interval of the fitted scale.
    From \eqref{eqn:circumradius-d-tail} we find the Fr\'echet limit law with shape parameter $d-1$
    (cf.\ \cite{Resnick:1987}*{Prop.~1.11}): for independent
    copies $\varrho_1,\varrho_2,\dotsc$ of $\varrho([B^d]_3)$,
    \begin{align*}
        \lim_{k\to\infty}
        \PP \Bigl(\max_{i\leq k}\varrho_i\leq (k\,C_d)^{\frac1{d-1}}\,x\Bigr)
         & = \lim_{k\to\infty} \biggl(1-\PP\Bigl(\varrho([B^d]_3)> (k\,C_d)^{\frac{1}{d-1}}\,x\Bigr)\biggr)^k \\
         & = \lim_{k\to\infty} \Bigl(1-\frac{1}{kx^{d-1}}\, \bigl(1+o(1)\bigr)\Bigr)^k
        = e^{-x^{-(d-1)}},
    \end{align*}
    for $x>0$.

    Combining the two descriptions, the law of $\varrho([B^2]_3)$ is explicit on
    $[1,\infty)$ and, on $(0,1)$, only on the containment event, where
    \[
        \PP\bigl(\varrho([B^2]_3)\leq \min\{x,1-\|\bc([B^2]_3)\|\}\bigr) =24\,\ibeta(4,3;x).
    \]
    The law of $\varrho([B^2]_3)$ on the
    remaining event $\{\varrho([B^2]_3)<1<\varrho([B^2]_3)+\|\bc([B^2]_3)\|\}$, of probability
    $4-\frac{32}{\pi^2}-\frac25\approx35.77\%$, appears to be open.
\end{remark}

\subsection{The proof of Theorem \ref{thm:sylvester-general}}%\label{sec:correction}

Recall that $K\subset\RR^2$ is an $r$-ball-convex body with $0<r\leq R$, and $\bX_1,\bX_2,\bX_3$ are
independent uniform random points in $K$.
The starting point is the decomposition of $\EE\vol_2([K]_3^R)$ provided by
Lemma~\ref{lem:structure}.

\begin{lemma}%\label{lem:master}
    For every $R\geq r$, with $\varrho([K]_3)$ the circumradius, $L_{ij}:=\tfrac12\|\bX_i-\bX_j\|$ and
    $L_{\max}:=\max_{i<j}L_{ij}$,
    \begin{equation}\label{eqn:master}
        \EE\vol_2\bigl([K]_3^R\bigr)
        = \EE\vol_2\bigl([K]_3\bigr) + \frac{3}{2}\,\EE\vol_2\bigl([K]_2^R\bigr) - E_R,
    \end{equation}
    where
    \begin{equation}\label{eqn:Edef}
        E_R := \EE\biggl[\Bigl(\vol_2([K]_3) + \frac{1}{2}\sum_{1\leq i<j\leq 3} V_R(L_{ij}) - V_R(L_{\max})\Bigr)
            \bone\{\varrho([K]_3)>R\}\biggr].
    \end{equation}
\end{lemma}

\begin{proof}
    By Lemma~\ref{lem:structure} we almost surely have
    \begin{align*}
        \vol_2\bigl([K]_3^R\bigr)
         & = \Bigl(\vol_2([K]_3)+\frac12\sum_{i<j}V_R(L_{ij})\Bigr)                                                    \\
         & \qquad {}- \Bigl(\vol_2([K]_3)+\frac12\sum_{i<j}V_R(L_{ij}) - V_R(L_{\max})\Bigr)\bone\{\varrho([K]_3)>R\}.
    \end{align*}
    This uses Lemma~\ref{lem:structure} for points in a common $R$-disc, which
    the three points of $K$ span since $K$ is $R$-ball-convex.  Take
    expectations and use $\EE V_R(L_{ij}) = \EE\vol_2([K]_2^R)$ for each of
    the three pairs.
\end{proof}

The first leading term of \eqref{eqn:master} is the classical Sylvester volume
$\EE\vol_2([K]_3)$ and the second is the ball-segment moment $\EE\vol_2([K]_2^R)$.
The point of \eqref{eqn:master} is that $E_R$ is supported on the event
$\{\varrho([K]_3)>R\}$ that the circumradius exceeds $R$.

The next lemma collects the only two integrals over a ball-segment that we shall
need.

\begin{lemma}%\label{lem:segment-moments}
    Let $0\leq\ell\leq R$, let $\conv_R\{\bp,\bq\}$ be the $R$-ball-segment
    spanned by two points $\bp$ and $\bq$ at distance $2\ell=\|\bp-\bq\|$.
    With $h$ and $\gamma$ defined as before, we have
    \begin{align}
        W_R(\ell)    & := \int_{\conv_R\{\bp,\bq\}}
        \operatorname{dist}\bigl(\bx,\operatorname{aff}\{\bp,\bq\}\bigr)\dint\bx
        = 2\ell\bigl(R^2-\tfrac{1}{3}\ell^2\bigr) - 2R^2h\gamma ,
        \label{eqn:W}                                                                            \\
        \Psi_R(\ell) & := \int_{\conv_R\{\bp,\bq\}} V_R\Bigl(\tfrac12\|\bx-\bp\|\Bigr)\,\dint\bx
        = 5 R^2\ell^2 - \ell^4 - 4R^2\ell h\gamma - R^4\gamma^2 .
        \label{eqn:Psi}
    \end{align}
\end{lemma}

\begin{proof}
    Place $\bp=-\ell\be_1$, $\bq=\ell\be_1$, so that
    $\conv_R\{\bp,\bq\}=B^2(\bc_-,R)\cap B^2(\bc_+,R)$ with
    $\bc_\pm=\pm h\be_2$, as in the proof of Lemma~\ref{lem:structure}.

    First, we prove \eqref{eqn:W}. By symmetry, $W_R(\ell)=4\int_0^{R-h} y\sqrt{R^2-(y+h)^2}\,\dint y$,
    since for $y\in[0,R-h]$ the horizontal slice of the upper half of the
    ball-segment has length $2\sqrt{R^2-(y+h)^2}$.  Substituting $u=y+h$,
    \begin{equation*}
        W_R(\ell) = 4\int_h^R (u-h)\sqrt{R^2-u^2}\,\dint u
        = \frac43 \ell^3
        - 4h\biggl(\frac{R^2}{2}\gamma - \frac{\ell h}{2}\biggr).
    \end{equation*}
    This gives \eqref{eqn:W} after substituting $h^2=R^2-\ell^2$.

    In order to prove \eqref{eqn:Psi}, we use polar coordinates $\bx=\bp+t\,\bu(\theta)$ centred at
    the vertex $\bp$, with $\bu(\theta)=(\cos\theta,\sin\theta)$ measured from
    the direction of $\bq-\bp$.  The tangent of $\partial B^2(\bc_-,R)$ at $\bp$
    has direction $(\cos\gamma,\sin\gamma)$, so the ball-segment is
    $\{\bp+t\,\bu(\theta):\ |\theta|\leq\gamma,\ 0\leq t\leq \tau(\theta)\}$,
    and writing $\bv:=\bp-\bc_-$ (so $\|\bv\|=R$), the equation
    $\|\bv+t\,\bu(\theta)\|=R$ gives
    $\tau(\theta)=-2 \langle \bv,\bu(\theta)\rangle = 2R\sin(\gamma-\theta)$ for
    $\theta\in[0,\gamma]$.  Hence, by symmetry and $\theta\mapsto\gamma-\theta$,
    \begin{equation*}
        \Psi_R(\ell) = 2\int_0^\gamma\int_0^{2R\sin\theta} V_R\Bigl(\frac{t}{2}\Bigr)\, t\,\dint t\,\dint\theta .
    \end{equation*}
    Substituting $t=2R\sin\nu$, so that $V_R(t/2)=V_R(R\sin\nu)=R^2(2\nu-\sin2\nu)$ and
    $t\,\dint t=2R^2\sin2\nu\,\dint\nu$ by \eqref{eqn:segment-2d},
    \begin{align*}
        \int_0^{2R\sin\theta} V_R\Bigl(\frac{t}{2}\Bigr)t\,\dint t
         & = 2R^4\int_0^\theta(2\nu-\sin2\nu)\sin2\nu\,\dint\nu                                                \\
         & = 2R^4\biggl(\frac{\sin2\theta}{2}-\theta\cos2\theta-\frac{\theta}{2}+\frac{\sin4\theta}{8}\biggr),
    \end{align*}
    and integrating once more over $\theta\in[0,\gamma]$ gives
    \begin{equation*}
        \Psi_R(\ell) = 4R^4\biggl(\frac{1-\cos2\gamma}{2}-\frac{\gamma\sin2\gamma}{2}
        -\frac{\gamma^2}{4}+\frac{1-\cos4\gamma}{32}\biggr),
    \end{equation*}
    which is \eqref{eqn:Psi} after substituting $\sin\gamma=\ell/R$,
    $\cos\gamma=h/R$.
\end{proof}

We are now in a position to calculate $E_R$, and to prove the second main result.

\begin{proof}[Proof of Theorem~\ref{thm:sylvester-general}]
    We claim that
    \begin{equation}\label{eqn:E}
        E_R = \frac{3}{2\vol_2(K)}\,\EE\biggl[\frac{32}{3} L_K^4
            - 8\,\Delta_R(L_K)\,V_R\bigl(L_K\bigr) - \frac32 V_R\bigl(L_K\bigr)^2\biggr],
    \end{equation}
    a one-dimensional integral against the law of $L_K$ in which no explicit power
    of $R$ occurs, or equivalently
    \begin{equation}\label{eqn:E-split}
        E_R = \frac{16}{\vol_2(K)}\,\EE L_K^4
        - \frac{9}{4\vol_2(K)}\,\EE V_R\bigl(L_K\bigr)^2
        - \frac{12}{\vol_2(K)}\,\EE\bigl[\Delta_R(L_K)\,V_R(L_K)\bigr],
    \end{equation}
    with $\Delta_R(\ell)=\ell\sqrt{R^2-\ell^2}$ as in \eqref{eqn:def-Delta-R}.

    By Lemma~\ref{lem:structure}(i) the three events
    $A_j:=\{\bX_j\in\conv_R\{\bX_i,\bX_k\}\}$, $\{i,j,k\}=\{1,2,3\}$, are
    pairwise disjoint (a triangle has at most one angle $\geq\pi/2$) and
    $\{\varrho([K]_3)>R\}=A_1\cup A_2\cup A_3$ up to a set of measure zero.  On $A_3$ the longest
    side is $L_{\max}=L_{12}$, by Lemma~\ref{lem:structure}(iii).
    Moreover,
    $K$ is $r$-ball-convex and $R\geq r$, so $\conv_R\{\bx_1,\bx_2\}\subset K$ for
    all $\bx_1,\bx_2\in K$. Hence, after conditioning on two of the points, the
    remaining integration over the ball-segment is never clipped by $\partial K$, and
    the conditional density of the third point on the ball-segment is
    $1/\vol_2(K)$.
    We treat the three parts of \eqref{eqn:Edef} one by one. All three use the
    exchangeability of $\bX_1,\bX_2,\bX_3$.

    \emph{(a)} Conditionally on $\bX_1,\bX_2$, we have
    $\vol_2([K]_3)=L_{12}\operatorname{dist}(\bX_3,\operatorname{aff}\{\bX_1,\bX_2\})$,
    so integrating $\bX_3$ over $\conv_R\{\bX_1,\bX_2\}$ and using \eqref{eqn:W},
    \begin{equation*}
        \EE\bigl[\vol_2([K]_3)\bone\{\varrho([K]_3)>R\}\bigr]
        = 3\,\EE\bigl[\vol_2([K]_3)\bone_{A_3}\bigr]
        = \frac{3}{\vol_2(K)}\,\EE\bigl[L_K\,W_R(L_K)\bigr].
    \end{equation*}

    \emph{(b)} Again by disjointness, and since $\bX_2$ is a vertex of
    $\conv_R\{\bX_2,\bX_3\}$,
    \begin{align*}
        \EE\bigl[V_R(L_{12})\bone\{\varrho([K]_3)>R\}\bigr]
         & = \EE\bigl[V_R(L_{12})\bone_{A_3}\bigr] + \EE\bigl[V_R(L_{12})\bone_{A_1}\bigr]
        + \EE\bigl[V_R(L_{12})\bone_{A_2}\bigr].
    \end{align*}
    Since $V_R(L_{12}) = \vol_2(\conv_R\{\bX_1,\bX_2\})$,
    \begin{align*}
        \EE\bigl[V_R(L_{12})\bone_{A_3}\bigr]
         & =
        \frac{1}{\vol_2(K)^3}
        \int_{K}\int_{K}\int_{\conv_R\{\bx_1,\bx_2\}} V_R\bigl(\tfrac{1}{2}\|\bx_1-\bx_2\|\bigr) \, \dint\bx_3\dint\bx_2\dint\bx_1 \\
         & =\frac{1}{\vol_2(K)}\EE\bigl[V_R(L_{12})^2\bigr],
    \end{align*}
    and by \eqref{eqn:Psi}
    \begin{align*}
        \EE\bigl[V_R(L_{12})\bone_{A_1}\bigr]
         & = \frac{1}{\vol_2(K)^3}
        \int_{K}\int_{K}\int_{\conv_R\{\bx_2,\bx_3\}} V_R\bigl(\tfrac{1}{2}\|\bx_1-\bx_2\|\bigr) \, \dint\bx_1\dint\bx_2\dint\bx_3 \\
         & = \frac{1}{\vol_2(K)}\EE\bigl[\Psi_R(L_{23})\bigr].
    \end{align*}
    Since the three pairs are exchangeable,
    \begin{align*}
        \EE\Bigl[\frac12\sum_{i<j}V_R(L_{ij})\bone\{\varrho([K]_3)>R\}\Bigr]
         & = \frac32\,\EE\bigl[V_R(L_{12})\bone\{\varrho([K]_3)>R\}\bigr]    \\
         & = \frac{3}{\vol_2(K)} \biggl(\frac{1}{2}\EE\bigl[V_R(L_K)^2\bigr]
        + \EE\bigl[\Psi_R(L_K)\bigr]\biggr).
    \end{align*}

    \emph{(c)} Since $L_{\max}=L_{12}$ on $A_3$,
    \begin{equation*}
        \EE\bigl[V_R(L_{\max})\bone\{\varrho([K]_3)>R\}\bigr] = 3\,\EE\bigl[V_R(L_{12})\bone_{A_3}\bigr]
        = \frac{3}{\vol_2(K)}\,\EE\bigl[V_R(L_K)^2\bigr].
    \end{equation*}

    Adding (a) and (b) and subtracting (c) gives
    \begin{equation}\label{eqn:E-three}
        E_R = \frac{3}{2\vol_2(K)}\,\EE\bigl[2L_K\,W_R(L_K) + 2\Psi_R(L_K) - V_R(L_K)^2\bigr].
    \end{equation}
    Finally, by \eqref{eqn:W}, \eqref{eqn:Psi} and
    $V_R(\ell)=2R^2\gamma-2\ell h$ of \eqref{eqn:segment-2d}, both sides of the
    algebraic identity
    \begin{equation}\label{eqn:key-identity}
        2\ell\, W_R(\ell) + 2\Psi_R(\ell) - V_R(\ell)^2
        = \frac{32}{3}\ell^4 - 8\ell h\,V_R(\ell) - \frac32 V_R(\ell)^2
    \end{equation}
    are equal to
    \begin{equation*}
        10R^2\ell^2 + \frac23\ell^4 - 4R^2\ell h\,\gamma - 6R^4\gamma^2,
    \end{equation*}
    which can be checked by direct substitution.
    Applied at $\ell=L_K$,
    \eqref{eqn:key-identity} turns \eqref{eqn:E-three} into \eqref{eqn:E}.
    Substituting \eqref{eqn:E-split} into
    \eqref{eqn:master} gives \eqref{eqn:intro-vol3}.

    Furthermore, \eqref{eqn:pn2} with $n=4$ gives
    $p_{4,2}^R(K)=\frac{6}{\vol_2(K)^2}\,\EE V_R(L_K)^2$, and \eqref{eqn:pn3}
    with $n=4$ gives
    $p_{4,3}^R(K)=\frac{4}{\vol_2(K)}\EE\vol_2([K]_3^R)-2p_{4,2}^R(K)$.
    Inserting \eqref{eqn:intro-vol3} yields the stated expression, the
    coefficient of $\EE V_R(L_K)^2$ being
    $\frac{9}{\vol_2(K)^2}-\frac{12}{\vol_2(K)^2}=-\frac{3}{\vol_2(K)^2}$.
    Finally, $p_{3,2}^R(K)=\frac{3}{\vol_2(K)}\,\EE V_R(L_K)$ and $p_{3,3}^R(K)=1-p_{3,2}^R(K)$
    are \eqref{eqn:pn2} and \eqref{eqn:pnn} with $n=3$, while
    $p_{4,4}^R(K)=1-p_{4,3}^R(K)-p_{4,2}^R(K)$.
\end{proof}

The main takeaway of \eqref{eqn:E-split} is that determining $E_R$ is not harder than the
two-point moments: in addition to $\EE V_R(L_K)^2$ and $\EE L_K^4$, the only
new quantity is the single one-dimensional integral $\EE[\Delta_R(L_K)V_R(L_K)]$,
of exactly the same type.  It is worth emphasising that the explicit powers of $R$
carried by $W_R$ and $\Psi_R$ cancel in \eqref{eqn:key-identity}, so that the
entire dependence of $E_R$ on the radius is carried by $\Delta_R$ and $V_R$.
This is the reason why no power of $R$ appears in
Theorem~\ref{thm:sylvester-general}.

\subsection{Explicit values for the unit disc}%\label{sec:disc-values}

Specialising $K$ to the unit disc, where $\vol_2(B^2)=\pi$,
$\EE\vol_2([B^2]_3)=35/48\pi$ and $\EE L_{B^2}^4=5/48$, and
writing $m_1,m_2$ as in \eqref{eqn:m1m2} and
$\mu(R):=\EE\bigl[\Delta_R(L_{B^2})\,V_R(L_{B^2})\bigr]$, we obtain
\begin{equation}\label{eqn:vol3-explicit}
    \EE\vol_2\bigl([B^2]_3^R\bigr)
    = -\frac{15}{16\pi} + \frac{3}{2}\,m_1(R) + \frac{9}{4\pi}\,m_2(R)
    + \frac{12}{\pi}\,\mu(R),
\end{equation}
and correspondingly
\begin{align*}
    p_{4,2}^R(B^2) & = \frac{6}{\pi^2}\,m_2(R),                                         \\
    p_{4,3}^R(B^2) & = -\frac{15}{4\pi^2}+\frac{6}{\pi}\,m_1(R)-\frac{3}{\pi^2}\,m_2(R)
    +\frac{48}{\pi^2}\,\mu(R),                                                          \\
    p_{4,4}^R(B^2) & = 1-p_{4,3}^R(B^2)-p_{4,2}^R(B^2).
\end{align*}
This proves all of
Theorem~\ref{thm:sylvester-r-convex} except for the explicit values at $R=1$,
which we shall deal with next, and the asymptotics \eqref{eqn:intro-p44-rate},
which are established in Remark~\ref{rem:asymptotics}.

Note that it is enough to consider the unit disc $B^2$ since by homogeneity $f_0([rB^2]_n^R)$ and $f_0([B^2]_n^{R/r})$ have the same distribution and
$p_{n,k}^R(rB^2)=p_{n,k}^{R/r}(B^2)$ for all $0<r\leq R$.

\begin{corollary}\label{cor:E1}
    The mixed moment at the critical radius is
    \begin{equation}\label{eqn:mu-at-one}
        \mu(1) = \EE\bigl[\Delta_1(L_{B^2})V_1(L_{B^2})\bigr] = \frac{2\pi^2-17}{24},
    \end{equation}
    and consequently
    \begin{equation*}
        \EE\vol_2\bigl([B^2]_3^1\bigr) = \frac{87}{4\pi}-2\pi.
    \end{equation*}
    This, in particular, determines the probabilities $(p_{4,2}^1(B^2),p_{4,3}^1(B^2),p_{4,4}^1(B^2))$ as stated in Theorem \ref{thm:sylvester-r-convex}.
\end{corollary}

\begin{proof}
    Substitute $\ell=\sin\alpha$, $\alpha\in[0,\frac\pi2]$, in \eqref{eqn:E-split}.
    By \eqref{eqn:dist-density-2d} with $r=1$, the density $f_{2,1}$ of
    $L_{B^2}$ then transforms as
    $f(\ell)\dint\ell=\frac{4}{\pi}(\sin2\alpha)(\pi-2\alpha-\sin2\alpha)\dint\alpha$,
    and for $R=1$ we have $\sqrt{1-\ell^2}=\cos\alpha$, so that
    $\Delta_1(\sin\alpha)=\tfrac12\sin2\alpha$ and
    $V_1(\sin\alpha)=2\alpha-\sin2\alpha$. Hence,
    \begin{equation*}
        \mu(1) = \frac{2}{\pi}\int_0^{\frac{\pi}{2}}
        \bigl(\sin^2 2\alpha\bigr)\bigl(2\alpha-\sin2\alpha\bigr)
        \bigl(\pi-2\alpha-\sin2\alpha\bigr)\dint\alpha
        = \frac{2\pi^2-17}{24},
    \end{equation*}
    an elementary integral. This is \eqref{eqn:mu-at-one}.  With
    $\vol_2(B^2)=\pi$ and $\EE L_{B^2}^4=\frac{5}{48}$ the first term of \eqref{eqn:E-split} is
    $\frac{16}{\pi}\cdot\frac{5}{48}=\frac{5}{3\pi}$, and with
    $m_2(1)=\EE\vol_2([B^2]_2^1)^2=\frac{81-8\pi^2}{12}$ from
    \eqref{eqn:m1m2-at-one}, \eqref{eqn:E-split} gives
    \begin{align*}
        E_1 & = \frac{5}{3\pi}
        - \frac{9}{4\pi}\cdot\frac{81-8\pi^2}{12}
        - \frac{12}{\pi}\cdot\frac{2\pi^2-17}{24}
        = \frac{\pi}{2}-\frac{241}{48\pi}.
    \end{align*}
    From this, \eqref{eqn:master} and $\EE\vol_2([B^2]_2^1)=\frac{32}{3\pi}-\pi$,
    \begin{equation*}
        \EE\vol_2\bigl([B^2]_3^1\bigr)
        = \frac{35}{48\pi} + \frac{16}{\pi}-\frac{3\pi}{2}
        - \frac{\pi}{2}+\frac{241}{48\pi}
        = \frac{276}{48\pi}+\frac{16}{\pi}-2\pi = \frac{87}{4\pi}-2\pi.
    \end{equation*}
    The value of $p_{4,2}^1$ is Corollary~\ref{cor:pn2}.  By the Efron--Buchta
    identity,
    \[
        p_{4,3}^1=\frac4\pi\EE\vol_2([B^2]_3^1)-2p_{4,2}^1
        =\frac{87}{\pi^2}-8-\frac{81}{\pi^2}+8=\frac{6}{\pi^2},
    \]
    and $p_{4,4}^1=1-p_{4,3}^1-p_{4,2}^1=5-\frac{93}{2\pi^2}$.
\end{proof}

\section{The three moments \texorpdfstring{$m_1$, $m_2$ and $\mu$}{m1, m2 and mu}}\label{sec:moments}

In Section~\ref{sec:three}, the vertex distributions of three and four
uniform random points in the unit disc have been reduced to the three one-dimensional integrals
\begin{equation*}
    m_1(R) = \EE V_R(L_{B^2}),\qquad
    m_2(R) = \EE V_R(L_{B^2})^2,\qquad
    \mu(R) = \EE\bigl[\Delta_R(L_{B^2})\,V_R(L_{B^2})\bigr]
\end{equation*}
against the law of the half-distance $L_{B^2}$ of two random points, which has the explicit density \eqref{eqn:dist-density-2d}.

At the critical
radius $R=1$ all three moments are elementary, according to \eqref{eqn:k-moment-at-one} and
Corollary~\ref{cor:E1}.
Our aim in this section is to show that for $R>1$ they can be expressed in a closed form using special functions.

\subsection{The three moments as rational series}%\label{sec:direct}

Let
\begin{equation*}%\label{eqn:S-def}
    u := \frac{\ell}{R}\in[0,1],
    \qquad
    S(u) := (\arcsin u)^2.
\end{equation*}
We use the following two expansions
\begin{align}
    \arcsin u-u\sqrt{1-u^2}
     & = 2\sum_{n\geq0}\frac{1}{4^{n}}\binom{2n}{n}\frac{u^{2n+3}}{2n+3},
    \label{eqn:arcsin-exp}                                                \\
    S(u)
     & = \frac12\sum_{n\geq1}\frac{4^n}{n^2}\binom{2n}{n}^{\!-1} u^{2n},
    \label{eqn:arcsin2-exp}
\end{align}
as well as the classical identity
\begin{equation}\label{eqn:binomial-Gamma-rel}
    \frac{1}{4^n}\binom{2n}{n}
    = \frac{\Gamma(n+\frac12)}{\sqrt{\pi}\,\Gamma(n+1)},
\end{equation}
for $n\in \NN_0$, where $\Gamma$ is Euler's Gamma function.

The first expansion \eqref{eqn:arcsin-exp} follows from the classical binomial series
\begin{equation*}
    \frac{1}{\sqrt{1-z^2}} = \sum_{n\geq 0} \frac{1}{4^n} \binom{2n}{n}\, z^{2n},
\end{equation*}
for $|z|<1$, since
\begin{equation*}
    \arcsin u-u\sqrt{1-u^2}
    = 2\int_{0}^u \frac{z^2}{\sqrt{1-z^2}} \, \dint z.
\end{equation*}
One directly verifies that \eqref{eqn:arcsin-exp} also converges for $u=1$.

The second expansion \eqref{eqn:arcsin2-exp} is also very classical, see \cite{BC:2007}*{Equation (2.1)} and its historical discussion there, and follows, for example, by noticing that $S(u)$ is the unique solution
of the following second order differential equation
\begin{equation*}
    (1-u^2) S'' (u) - u S'(u) = 2, \qquad \text{with $S(0)=S'(0)=0$.}
\end{equation*}

Using the identity
\begin{equation*}%\label{eqn:cross}
    u\sqrt{1-u^2}\,\arcsin u = \frac{u(1-u^2)}{2}\,S'(u),
\end{equation*}
and by \eqref{eqn:segment-2d} and \eqref{eqn:def-Delta-R}, we find
\begin{align}
    V_R(\ell)               & = 2R^2\bigl(\arcsin u-u\sqrt{1-u^2}\bigr),
    \label{eqn:V-in-u}                                                   \\
    V_R(\ell)^2             & = 4R^4\bigl(S(u)-u(1-u^2)(S'(u)-u)\bigr),
    \label{eqn:V2-in-u}                                                  \\
    \Delta_R(\ell)V_R(\ell) & = R^4u(1-u^2)\bigl(S'(u)-2u\bigr).
    \label{eqn:DV-in-u}
\end{align}
Thus all three moments can be expressed using the series expansions \eqref{eqn:arcsin-exp} and \eqref{eqn:arcsin2-exp}.

\begin{proposition}[the three moments as rational series]\label{prop:direct}
    For every $R\geq1$,
    \begin{align}
        m_1(R) & = \frac{512}{\pi}\sum_{n\geq0}
        \frac{(n+1)(n+2)}{(2n+1)(2n+3)^2(2n+5)^2(2n+7)}\, \frac{1}{R^{2n+1}},
        \label{eqn:m1-direct}                   \\
        m_2(R) & = 4\sum_{n\geq1}
        \frac{(2n+3)(2n+5)}{(n+1)(n+2)^2(n+3)^2(n+4)} \, \frac{1}{R^{2n}},
        \label{eqn:m2-direct}                   \\
        \mu(R) & = \frac{5}{36}-\sum_{n\geq1}
        \frac{(2n+5)}{(n+1)(n+2)(n+3)^2(n+4)} \, \frac{1}{R^{2n}}.
        \label{eqn:mu-direct}
    \end{align}
\end{proposition}

\begin{proof}
    Inserting \eqref{eqn:arcsin-exp} into \eqref{eqn:V-in-u} yields
    \begin{equation*}
        m_1(R) = \frac{4}{R}\sum_{n\geq0}\frac{1}{2n+3}
        \binom{2n}{n}\frac{\EE L_{B^2}^{2n+3}}{4^{n}R^{2n}} .
    \end{equation*}
    By \eqref{eqn:L-moments} we have
    $\EE L_{B^2}^{2n+3}=16\,\Gamma(n+3)\bigl/\bigl(\sqrt\pi\,(2n+5)(2n+7)
        \Gamma(n+\frac72)\bigr)$. Using \eqref{eqn:binomial-Gamma-rel}, the two gamma quotients that remain are
    \begin{equation*}
        \frac{\Gamma\bigl(n+\frac12\bigr)}{\Gamma\bigl(n+\frac72\bigr)}
        = \frac{8}{(2n+1)(2n+3)(2n+5)},
        \qquad
        \frac{\Gamma(n+3)}{\Gamma(n+1)} = (n+1)(n+2).
    \end{equation*}
    Collecting the constants gives \eqref{eqn:m1-direct}.

    Set $c_n:=\frac{4^n}{2n^2} \binom{2n}{n}^{-1}$, so that
    by \eqref{eqn:arcsin2-exp} $S(u) = \sum_{n\geq 1} c_n u^{2n}$ and
    note that for $n\geq 2$
    \begin{equation*}
        (2n)\, c_n - 2(n-1)\,c_{n-1} = - \frac{n}{n-1}\, c_n.
    \end{equation*}
    Therefore
    \begin{align*}
        u(1-u^2)S'(u)
         & = 2c_1u^2 + \sum_{n\geq 2} \bigl(2nc_n-2(n-1)c_{n-1}\bigr) u^{2n}
        = 2u^2 - \sum_{n\geq 2} \frac{n}{n-1}\, c_n \, u^{2n}.
    \end{align*}
    Thus, by \eqref{eqn:V2-in-u} and \eqref{eqn:DV-in-u},
    \begin{align*}
        V_R(\ell)^2             & = 4R^4 \sum_{n\geq 3} \frac{2n-1}{n-1}\,c_n\, u^{2n},                         &
        \Delta_R(\ell)V_R(\ell) & = R^4 \Bigl(\frac{4}{3}u^4 - \sum_{n\geq 3} \frac{n}{n-1}\, c_n u^{2n}\Bigr).
    \end{align*}
    For $u=L_{B^2}/R$ and taking expectations we find
    \begin{align*}
        m_2(R) & = 4 \sum_{n\geq 3} \frac{2n-1}{n-1}\, \frac{c_n\, \EE L_{B^2}^{2n}}{R^{2(n-2)}},                        &
        \mu(R) & = \frac{4}{3} \EE L_{B^2}^4 - \sum_{n\geq 3} \frac{n}{n-1}\, \frac{c_n\, \EE L_{B^2}^{2n}}{R^{2(n-2)}}.
    \end{align*}

    Using \eqref{eqn:L-moments} and \eqref{eqn:binomial-Gamma-rel} we find $\frac{4}{3} \EE L_{B^2}^4 = \frac{5}{36}$ and
    \begin{equation*}
        c_n\, \EE L_{B^2}^{2n}
        = \frac{\sqrt{\pi} \, \Gamma(n+1)}{2n^2\, \Gamma\bigl(n+\frac{1}{2}\bigr)}
        \cdot
        \frac{4 \, \Gamma\bigl(n+\frac{3}{2}\bigr)}{\sqrt{\pi} (n+1)\, \Gamma(n+3)}
        = \frac{2n+1}{n^2(n+1)^2(n+2)}.
    \end{equation*}
    This completes the proof of \eqref{eqn:m2-direct} and \eqref{eqn:mu-direct} after shifting the index from $n$ to $n+2$.
\end{proof}

\subsection{Reduction to dilogarithms}

We now show that \eqref{eqn:m1-direct}--\eqref{eqn:mu-direct} can be expressed in closed form using the following functions for $R\geq1$:
\begin{align*}
    \cL_1(R) & := (R^2-1)\,\log\Bigl(1-\frac1{R^2}\Bigr),
             &
    \cL_2(R) & := \Li_2\Bigl(\frac1{R^2}\Bigr) = \int_{0}^{1} \frac{1}{z}\, \ln\Bigl(\frac{R^2}{R^2-z}\Bigr)\,\dint z,
    \intertext{and}
    \cX_1(R) & := (R^2-1)\,\artanh\frac1R,
             &
    \cX_2(R) & := \chi_2\Bigl(\frac1R\Bigr) = \int_{0}^1 \frac{1}{z} \artanh\Bigl(\frac{z}{R}\Bigr) \, \dint z.
\end{align*}
Here $\Li_2$ is the classical dilogarithm function and $\chi_2$ is the Legendre chi function and we have $\chi_2(z)=\frac{1}{2} \bigl(\Li_2(z)-\Li_2(-z)\bigr)$.
Their series expansions are
\begin{align}\label{eqn:L-series}
    \cL_1(R) & = -1 + \sum_{n\geq 1} \frac{1}{n(n+1)}\, \frac{1}{R^{2n}},                     &
    \cL_2(R) & = \sum_{n\geq 1} \frac{1}{n^2} \, \frac{1}{R^{2n}},                              \\
    \cX_1(R) & = R-\sum_{n\geq 0} \frac{2}{(2n+1)(2n+3)}\, \frac{1}{R^{2n+1}},                &
    \cX_2(R) & = \sum_{n\geq 0} \frac{1}{(2n+1)^2}\, \frac{1}{R^{2n+1}}. \label{eqn:X-series}
\end{align}
Note that $\cL_1(1) =\cX_1(1)= 0$, $\cL_2(1) = \Li_2(1)= \frac{\pi^2}{6}$ and $\cX_2(1) = \chi_2(1) = \frac{\pi^2}{8}$.

\begin{proposition}[the three moments are dilogarithmic]\label{prop:dilog}
    For every $R\geq1$,
    \begin{align*}
        3\pi\,m_1(R)
         & = 15R^5+20R^3-3R
        \nonumber
        \\
         & \qquad -3\bigl(5R^4-2R^2+1\bigr)\,\cX_1(R)-12R^2\bigl(3R^2-1\bigr)\, \cX_2(R),
        \\[2pt] \nonumber
        12\, m_2(R)
         & = 60R^6+18R^4+8R^2-5
        \\ \nonumber
         & \qquad + 12R^2\bigl(5R^4-2R^2+1\bigr)\,\cL_1(R)-24R^4\bigl(3R^2-1\bigr)\,\cL_2(R),
        \\[2pt] \nonumber
        24\,\mu(R)
         & = -12R^6-10R^2+5
        \\
         & \qquad -6R^2\bigl(2R^4-R^2+1\bigr)\,\cL_1(R)
        +12R^6\,\cL_2(R).
    \end{align*}
\end{proposition}

\begin{proof}
    For the sequence of rational coefficients of the series expansions of $m_1(R)$ in Proposition~\ref{prop:direct}
    we have the partial fraction decomposition
    \begin{align*}
         & \frac{(n+1)(n+2)}{(2n+1)(2n+3)^2(2n+5)^2(2n+7)}                                                             \\
         & \qquad = \frac{1}{512}\Bigl(\frac{2}{(2n+1)(2n+3)} - \frac{4}{(2n+3)(2n+5)} + \frac{10}{(2n+5)(2n+7)}\Bigr) \\
         & \qquad\qquad \qquad + \frac{1}{128} \Bigl(\frac{1}{(2n+3)^2} - \frac{3}{(2n+5)^2}\Bigr).
    \end{align*}
    Summing over $n\in\NN_0$ and using the series expansions \eqref{eqn:X-series} we find
    \begin{align*}
        \pi\, m_1(R) & = \bigl(R-\cX_1(R)\bigr) - 2R^2\Bigl(R-\cX_1(R)-\frac{2}{3R}\Bigr)                                     \\
                     & \qquad + 5R^4\Bigl(R-\cX_1(R)-\frac{2}{3R}-\frac{2}{15R}\Bigr)                                         \\
                     & \qquad +4 R^2\Bigl(\cX_2(R)-\frac{1}{R}\Bigr) - 12 R^4\Bigl(\cX_2(R)-\frac{1}{R}-\frac{1}{9R^3}\Bigr).
    \end{align*}
    The other identifications follow similarly using \eqref{eqn:L-series}.
\end{proof}

\begin{remark}\label{rem:asymptotics}
    By  \eqref{eqn:m1-direct}--\eqref{eqn:mu-direct}:
    \begin{align*}
        m_1(R)
         & = \frac{1024}{1575\pi}\Biggl(\frac1R+\frac1{7R^3}+\cdots\Biggr), &
        m_2(R)
         & = \frac{7}{72} \Biggl(\frac1{R^2}+\frac{9}{25R^4}+\cdots\Biggr),   \\
        \mu(R)
         & = \frac5{36}-\frac7{480R^2}-\frac1{200R^4}-\cdots ,
    \end{align*}
    the three leading terms being $\frac{4}{3R}\, \EE L_{B^2}^3$, $A\frac{1}{R^2}$ with
    $A=\frac7{72}$ as in \eqref{eqn:m2-asymptotics}, and
    $\frac43\,\EE L_{B^2}^4=\frac5{36}$.  This extends \eqref{eqn:intro-p44-rate} to
    \begin{align*}%\label{eqn:p44-full}
        p_{4,4}^R(B^2)
         & = 1-\frac{35}{12\pi^2}
        -\frac{2048}{525\pi^2R}
        +\frac{49}{120\pi^2R^2}
        -\frac{2048}{3675\pi^2R^3}
        +\frac{27}{200\pi^2R^4}
        +O\!\bigl(R^{-5}\bigr),
        \intertext{and likewise}
        p_{3,3}^R(B^2)
         & =1-\frac{1024}{525\pi^2R}-\frac{1024}{3675\pi^2R^3}
        +O\!\bigl(R^{-5}\bigr),
    \end{align*}
    refining Corollary~\ref{cor:intro-circumradius}.
\end{remark}

\begin{remark}[hypergeometric representation]%\label{rem:3F2-at-one}
    In the notation of \cite{DLMF}*{Ch.~16} the three series of
    Proposition~\ref{prop:direct} can also be represented by
    \begin{equation*}
        m_1 = \frac{1024}{1575\pi R}\,
        {}_3F_2\Bigl(\tfrac12,\tfrac32,3;\,\tfrac72,\tfrac92;\,\tfrac1{R^2}\Bigr),
        \quad
        m_2 = \frac{7}{72R^2}\,
        {}_4F_3\Bigl(1,2,3,\tfrac92;\,\tfrac52,5,6;\,\tfrac1{R^2}\Bigr),
    \end{equation*}
    and
    \begin{equation*}
        \mu=\frac{5}{24}-\frac{5}{72}\,
        {}_4F_3\Bigl(1,1,3,\tfrac72;\,\tfrac52,4,5;\,\tfrac1{R^2}\Bigr).
    \end{equation*}
\end{remark}

\section{Extremal bodies at given area}\label{sec:extremal}

By the Efron--Buchta identities in Theorem \ref{thm:R-buchta} the distribution of $f_0([K]_n^R)$
is determined by the moments of $\vol_2([K]_k^R)$. Blaschke's proof for the extremality of the disc for the Sylvester probability $p_{4,4}(K)$ hinges on the fact that
\begin{equation}\label{eqn:p44-aff}
    4\Bigl(1- \frac{\EE\vol_2([K]_3)}{\vol_2(K)}\Bigr) = \EE f_0([K]_4) = 3+ p_{4,4}(K).
\end{equation}
Hence for the affine convex hull the extremality of $\EE\vol_2([K]_3)$ is equivalent to the extremality of $p_{4,4}(K)$ and $\EE f_0([K]_4)$.

For a convex body $K\subset \RR^d$ Groemer \cite{Groemer:1974} showed that $\EE\vol_d([K]_n)^k$ is minimised uniquely by
ellipsoids of the same volume as $K$ for all $n\geq d+1$ and $k\geq 1$, that is,
\[
    \EE\vol_d([K]_n)^k \geq \EE\vol_d([r_K B^d]_n)^k,
\]
where $r_K := (\frac{1}{\kappa_d}\vol_d(K))^{\frac{1}{d}}$ such that $\vol_d(K) = \vol_d(r_K B^d)$.
See also Pfiefer \cite{Pfiefer:1990}, who extended this result even further, in particular to compact sets $K\subset \RR^d$.

Our goal in this section is to establish a corresponding result for the $R$-ball-convex hull.

\subsection{Extremiser for two random points and circumradius}

First, there is an important distinction to be made between the affine convex hull and the $R$-ball-convex hull, that is,
the $R$-segment between two distinct points $\bx_1,\bx_2\in K$ of an $r$-ball-convex body $K\subset \RR^d$ with $0<r\leq R$ has strictly positive area.
Hence, we actually want to prove that $\EE\vol_d([K]_n^R)$ is minimised by $r_KB^d$ for all $n\geq 2$.

However, the case $n=2$ is already determined by a result of Pfiefer \cite{Pfiefer:1990}*{Thm.\ 2} in full generality, whose result implies that $\EE\Phi(L_K)$ is minimised by $\EE\Phi(L_{r_KB^d})$ whenever $\Phi$ is a strictly increasing function and $(\bx_1,\bx_2)\mapsto \Phi\bigl(\frac{1}{2}\|\bx_1-\bx_2\|\bigr)$ is integrable.
In our case we consider $\vol_d([K]_2^R)^k = \vol_d(\mathrm{Sg}_R(L_K))^k=\Phi(L_K)$ where $\Phi(\ell):=\vol_d(\mathrm{Sg}_R(\ell))^k$ is strictly increasing and continuous for $\ell\in[0,R]$. Applying Pfiefer's theorem we immediately obtain the following.

\begin{corollary}[isoperimetric inequalities for two points at a fixed area]\label{cor:isoperimetric}
    Let $K\subset\RR^d$ be an $r$-ball-convex body. Then
    for every $k\geq 1$ and every $R\geq r$,
    \begin{equation*}%\label{eqn:iso-moments}
        \EE\vol_d\bigl([K]_2^R\bigr)^k \;\geq\;
        \EE\vol_d\bigl([r_K B^d]_2^R\bigr)^k,
    \end{equation*}
    with equality if and only if $K$ is a Euclidean ball.
    Consequently, for all $n\geq3$ and $R\geq r$, by \eqref{eqn:pn2}
    \begin{equation*}%\label{eqn:iso-pn2}
        p_{n,2}^R(K) \geq p_{n,2}^R(r_K B^d),
    \end{equation*}
    and by \eqref{eqn:circumradius-d}
    \begin{equation*}%\label{eqn:iso-circumradius}
        \PP(\varrho([K]_3)>R) \;\geq\;
        \PP\Bigl(\varrho([B^d]_3)>\frac{R}{r_K}\Bigr).
    \end{equation*}
\end{corollary}

\subsection{Groemer-type inequality for ball-convex bodies in the plane}\label{sec:groemer}

Let $K\subset \RR^2$ be an $r$-ball-convex body and set $r_K=\sqrt{\frac{1}{\pi}\vol_2(K)}$.
For $n\geq 2$ and $k\geq 1$ we aim to show that $\EE\vol_2([K]_n^R)^k\geq \EE\vol_2([r_K B^2]_n^R)^k$ with equality if and only if $K$ is a disc.
For $n=3$, the limiting case $R\to \infty$ is the classical result of Blaschke, which he established using Steiner symmetrisation, see for example the presentation in \cite{Calka:2019}*{Sec.\ 3.1.3}. Steiner symmetrisation is also at the core of Groemer's extension \cite{Groemer:1974}.

For a unit vector $\bv\in\SS^1$, we write $S_{\bv} K$ for the Steiner symmetral of $K$ with respect to the line $\bv^\bot :=\{\bx \in \RR^2 : \langle\bx, \bv\rangle = 0\}$, see, for example, \cite{Sch14}*{Sec.\ 10.3}.
Steiner symmetrisation for ball-convex bodies appears to have been used first by Artstein-Avidan and Florentin \cite{AAF:2026}*{Sec.\ 4}. They show, in particular, that $S_{\bv} K$ is again $r$-ball-convex when $d=2$, see \cite{AAF:2026}*{Thm.\ 4.1}, and that for $d\geq 3$ this fails in general, see \cite{AAF:2026}*{Sec.\ 4.2}. This is the main obstruction when trying to extend our methods to dimensions $d\geq 3$. However, note that Theorem \ref{thm:R-groemer} for $n=2$ and $d\geq 3$ is exactly Corollary \ref{cor:isoperimetric}.

Since we may obtain $r_K B^2$ as a limit after a sequence of Steiner symmetrisations of $K$, our aim is to show the following.

\begin{proposition}[monotonicity under Steiner symmetrisation]\label{prop:steiner-monotone}
    With the same assumptions as in Theorem \ref{thm:R-groemer} and for $\bv\in\SS^1$, we have the following
    \begin{equation}\label{eqn:steiner-monotone}
        \EE\vol_2([K]_n^R)^k \geq \EE\vol_2([S_{\bv} K]_n^R)^k,
    \end{equation}
    with equality if and only if $K$ is symmetric with respect to the line orthogonal to $\bv$ that passes through the centroid of $K$.
\end{proposition}
\begin{proof}
    Assume w.l.o.g.\ that $\bv =\be_2$ and $K\subset rB^2$. For the compact interval
    \[
        I := \{t\in \RR : \text{there is $y\in\RR$ such that $(t,y) \in K$}\}
    \]
    we may write
    \begin{equation*}
        K = \{(t,y) : t\in I, f_-(t)\leq y \leq f_+(t)\},
    \end{equation*}
    where $f_+$ is concave and $f_-$ is convex on $I$ and both are uniquely determined by $K$. We set $g:= \frac{1}{2}(f_+-f_-)$ as the chord half-length and $m:=\frac{1}{2}(f_+ + f_-)$ as the midpoint curve. Then
    \begin{equation*}
        S_{\bv} K = \{(t,y) : t\in I, |y|\leq g(t) \},
    \end{equation*}
    and for the reflection $K^\dag$ of $K$ in $\bv^\bot$ we have
    \begin{equation*}
        K^\dag = \{(t,-y) : (t,y)\in K\} = \{(t,y) : t\in I, -f_+(t) \leq y \leq -f_-(t)\}.
    \end{equation*}

    Let $\bx_1,\dotsc,\bx_n\in K$ be arbitrary. Then $\bx_i=(t_i,y_i)$, where $y_i = m(t_i)+g(t_i)u_i$ for $t_i\in I$ and $u_i\in[-1,1]$ are uniquely determined for all $i\in\{1,\dotsc,n\}$.
    For $\bt\in I^n$ and $\bu\in [-1,1]^n$ we define
    \begin{equation*}
        \bm := (m(t_1),\dotsc,m(t_n))\in \RR^n \qquad \text{and} \qquad
        g\odot \bu := (g(t_1)u_1,\dotsc,g(t_n)u_n) \in \RR^n,
    \end{equation*}
    where $\odot$ denotes the coordinatewise or Hadamard product of vectors.

    For $\ba, \bb\in \RR^n$ define
    \[
        F_{\ba}(\bb) := \vol_2(\conv_R\{ (a_i,b_i) : i =1,\dotsc,n\})^k,
    \]
    so that
    \begin{equation*}
        F_{\bt}(\bm+g\odot\bu) = \vol_2(\conv_R\{\bx_1,\dotsc,\bx_n\})^k.
    \end{equation*}
    In particular, by Fubini's theorem, we find
    \begin{align*}
        \EE\vol_2([K]_n^R)^k         & = \int_{I^n} \biggl(\int_{[-1,1]^n} F_{\bt}(\bm + g\odot\bu)\, \dint\bu \biggr) \Bigl(\prod_{i=1}^n \frac{g(t_i)}{\vol_2(K)}\Bigr) \,\dint \bt,  \\
        \EE\vol_2([S_{\bv} K]_n^R)^k & = \int_{I^n} \biggl(\int_{[-1,1]^n} F_{\bt}(g\odot\bu)\, \dint\bu \biggr) \Bigl(\prod_{i=1}^n \frac{g(t_i)}{\vol_2(K)}\Bigr) \,\dint \bt,        \\
        \EE\vol_2([K^\dag]_n^R)^k    & = \int_{I^n} \biggl(\int_{[-1,1]^n} F_{\bt}(-\bm + g\odot\bu)\, \dint\bu \biggr) \Bigl(\prod_{i=1}^n \frac{g(t_i)}{\vol_2(K)}\Bigr) \,\dint \bt.
    \end{align*}
    Since $\EE\vol_2([K]_n^R)^k = \EE\vol_2([K^\dag]_n^R)^k$, inequality \eqref{eqn:steiner-monotone} follows if
    \begin{equation}\label{eqn:FT-convex}
        \frac{1}{2} \Bigl(F_{\bt}(\bm + g\odot \bu) + F_{\bt}(-\bm+g\odot\bu)\Bigr) \geq F_{\bt}(g\odot \bu).
    \end{equation}
    By \cite{AAF:2026}*{Thm.\ 3.7}, we find that, using their language,
    \[
        h(s) := \vol_2(\conv_R\{\bx_i + s\alpha_i \bv : i=1,\dotsc,n\}) \qquad \text{for $s\in [0,2]$},
    \]
    is convex, where $\alpha_i := - m(t_i)$ for $i=1,\dotsc,n$. Then
    \[
        h(0)^k = F_{\bt}(\bm+g\odot\bu),\quad
        h(1)^k = F_{\bt}(g\odot\bu),\quad
        h(2)^k = F_{\bt}(-\bm+g\odot\bu),
    \]
    and since $s\mapsto h(s)^k$ is convex, $h(1)^k \leq \frac{1}{2}(h(0)^k+h(2)^k)$ yields \eqref{eqn:FT-convex}. This verifies \eqref{eqn:steiner-monotone}.

    To finish the proof, we only need to show the equality statement of \eqref{eqn:steiner-monotone}. Note first that $K$ is symmetric about some line orthogonal to $\bv$ if and only if $m(t)$ is a constant function.
 So, assume that $t\mapsto m(t)$ is not constant.
    Denote by $E_{12}\subset I^n\times [-1,1]^n$ the set of configurations such that the $n$ points $\bz_i = (t_i,g(t_i)u_i)\in S_{\bv} K$, $i=1,\dotsc,n$, are in general position, all are $R$-extremal for $\conv_R\{\bz_1,\dotsc,\bz_n\}$, $\bz_1,\bz_2$ are adjacent and $m(t_1)\neq m(t_2)$. Then $E_{12}$ is open and non-empty and therefore has positive Lebesgue measure.
    For $(\bt,\bu)\in E_{12}$ set $\bx_i:=(t_i,m(t_i)+g(t_i)u_i)\in K$ and $\alpha_i:=-m(t_i)$ for $i=1,\dotsc,n$ and consider $M(s) := \conv_R\{\bx_i+s\alpha_i\bv : i=1,\dotsc,n\}$.
    Note that $M(1) = \conv_R\{\bz_1,\dotsc,\bz_n\}$ and $\bz_1,\bz_2$ determine an $R$-ball edge of $M(1)$ and since the combinatorial type of $M(s)$ is locally constant, there exists $\varepsilon>0$ such that $\bx_1+s\alpha_1\bv,\bx_2+s\alpha_2\bv$ also determines an $R$-ball edge of $M(s)$ for all $|s-1|<\varepsilon$.
    Furthermore, by \eqref{eqn:V-derivative}, $\ell\mapsto V_R(\ell)$ is strictly convex and therefore
    \[
        V_R(\ell_{12}(s))
        = V_R\Bigl(\frac{1}{2}\|(\bx_1-\bx_2) + s (m(t_2)-m(t_1))\bv\|\Bigr)
    \]
    is strictly convex for $s\in[0,2]$, since $m(t_1)\neq m(t_2)$ by assumption.
    As before $h(s) = \vol_2(M(s))$ is convex for $s\in[0,2]$, and in particular for $|s-1|<\varepsilon$
    \begin{equation*}
        h(s) = \vol_2(\conv\{\bx_i+s\alpha_i\bv: i=1,\dotsc,n\}) + \frac{1}{2}\sum_{\text{$i<j$: $(\bx_i,\bx_j)$ $R$-edge of $M(s)$}} V_R(\ell_{ij}(s)),
    \end{equation*}
    is strictly convex since $V_R(\ell_{12}(s))$ is a summand.
    Thus, in $E_{12}$, the inequality \eqref{eqn:FT-convex} is strict and therefore \eqref{eqn:steiner-monotone} is a strict inequality if $m$ is not constant. This shows that equality in \eqref{eqn:steiner-monotone} implies that $t\mapsto m(t)$ is constant and therefore $K$ is symmetric about the line through its centroid orthogonal to $\bv$, as claimed.
\end{proof}

\begin{proof}[Proof of Theorem \ref{thm:R-groemer}]
    Assume w.l.o.g.\ $K\subset rB^2$.
    There is a sequence $(\bv_i)_{i\in \NN_0}$ such that the sequence $(K_i)_{i\in \NN_0}$ defined by $K_0:= K$ and $K_{i} := S_{\bv_{i-1}} K_{{i-1}}$ for $i=1,\ldots$ converges to $r_K B^2$ in the Hausdorff metric, see for example \cite{AGM:2015}*{Thm.\ 1.1.16}.
    By Proposition \ref{prop:steiner-monotone} the sequence $\EE \vol_2([K_i]_n^R)^k$ is non-increasing,
    since $\vol_2(K_i) = \vol_2(K)$ for all $i\in\NN_0$.
    Furthermore $K_1 = S_{\bv_0} K \subset S_{\bv_0}(rB^2) = rB^2$. Thus, by induction, $K_i\subset rB^2$ and therefore $\EE\vol_2([K_i]_n^R)^k \leq \vol_2(rB^2)^k$ for all $i\in \NN_0$.
    Continuity of $(\bx_1,\ldots,\bx_n)\mapsto \vol_2(\conv_R\{\bx_1,\dotsc,\bx_n\})^k$ follows by, for example, \cite{AAF:2026}*{Prop.\ 3.8}.
    We estimate
    \begin{align*}
         & \Bigl|\EE\vol_2([K_i]_n^R)^k - \EE\vol_2([r_K B^2]_n^R)^k\Bigr|                                                      \\
         & \qquad= \frac{1}{\vol_2(K)^n} \Biggl| \int_{(K_i)^n\triangle (r_K B^2)^n} \vol_2(\conv_R\{\bx_1,\dotsc,\bx_n\})^k \,
        \dint\bx_1 \dotsc \dint\bx_n\Biggr|                                                                                     \\
         & \qquad\leq \frac{\vol_2(rB^2)^k}{\vol_2(K)^n}\, \vol_{2n}\Bigl((K_i)^n\triangle (r_K B^2)^n\Bigr)                    \\
         & \qquad\leq \frac{\vol_2(rB^2)^k}{\vol_2(K)^n}\, n \vol_2(rB^2)^{n-1}\, \vol_{2}\Bigl(K_i\triangle r_K B^2\Bigr),
    \end{align*}
    where $\triangle$ denotes the symmetric difference of sets and in the last inequality we used that
    \begin{equation*}
        (K_i)^n\triangle (r_K B^2)^n \subset \bigcup_{i=1}^n (rB^2)^{i-1}\times (K_i\triangle r_K B^2) \times (rB^2)^{n-i}.
    \end{equation*}
    Since $\lim_{i\to\infty} \vol_2(K_i\triangle r_K B^2) =0$, we conclude
    \begin{equation*}
        \lim_{i\to\infty} \EE\vol_2([K_i]_n^R)^k = \EE\vol_2([r_K B^2]_n^R)^k.
    \end{equation*}

    For the equality case, let $\bv\in\SS^1$ be arbitrary. Then
    \begin{equation*}
        \EE \vol_2([r_K B^2]_n^R)^k = \EE\vol_2([K]_n^R)^k \geq \EE\vol_2([S_{\bv}K]_n^R)^k \geq \EE \vol_2([r_K B^2]_n^R)^k,
    \end{equation*}
    and therefore, by the equality condition of Proposition \ref{prop:steiner-monotone}, $K$ is symmetric about a line orthogonal to $\bv$ that passes through the centroid of $K$. Since $\bv$ was arbitrary, this shows that $K$ is a disc.
\end{proof}

\subsection{Isoperimetric inequality for the ball-convex analogue of affine surface area}
\label{sec:c-affine}

Schütt, Werner and Yalikun \cite{SWY:2025} recently derived the $r$-ball-convex analogue of the affine surface area for an $r$-ball-convex body $K\subset \RR^d$
\begin{equation*}
    \as^r(K) = \int_{\partial K} \prod_{i=1}^{d-1}\bigl(\kappa_i(K,\bx)-r^{-1}\bigr)^{\frac{1}{d+1}} \, \dint\bx,
\end{equation*}
where $\kappa_i(K,\cdot)$ denote the principal curvatures of $\partial K$.
In the plane, $\as^r(K)$ also appeared earlier in \cite{FKV:2014}.

In \cite{SWY:2025}, the authors gave the inequality
\begin{equation}\label{eqn:SWY-ineq}
    \as^r(K) \leq \as(K) \leq \as(r_K B^d),
\end{equation}
which follows by monotonicity in $r$ and the classical affine isoperimetric inequality.
Artstein-Avidan and Chor \cite{AAC:2025} proved the inequality
\begin{equation}\label{eqn:AAC_ineq}
    \as^r(K) \leq \as^r\Bigl(\frac{dr}{d+1}B^d\Bigr),
\end{equation}
with equality if and only if $K=\frac{d}{d+1}B^d$.

In the case when $K\subset \RR^2$ is an $r$-ball-convex body of class $C^2_+$, it was shown in \cite{FKV:2014}*{Thm.\ 1.1}  that
\begin{equation*}
    \lim_{n\to \infty} n^{\frac{2}{3}} \left( \vol_2(K) - \EE\vol_2([K]_n^r) \right) = c_2 \vol_2(K)^{\frac{2}{3}} \as^r(K),
\end{equation*}
where $c_2 := (2/3)^{1/3}\, \Gamma(5/3)$, see also \cites{FPV:2020, FP:2025}.
By Theorem \ref{thm:R-groemer}, we find for all $n\geq 2$ that
\begin{equation*}
    \vol_2(K) - \EE\vol_2([K]_n^r) \leq \vol_2(r_K B^2) - \EE\vol_2([r_K B^2]_n^r),
\end{equation*}
which yields Corollary \ref{cor:R-affine-iso} after multiplying by $n^{2/3}$ and  taking limits as $n\to\infty$ on both sides.

Note that for $r\to\infty$ the right-hand side of \eqref{eqn:R-affine-iso} monotonically increases to $2\pi r_K^{\frac{2}{3}} = \as(r_K B^2)$, therefore giving the classical affine isoperimetric inequality in the limit. Furthermore, for $d=2$ we recover \eqref{eqn:SWY-ineq}
\begin{equation*}
    \as^r(K)\leq \as^r(r_KB^2)\leq \as(r_KB^2).
\end{equation*}
Moreover, since $s\in[0,r]\mapsto s^2 (1-\frac{s}{r})$ has a unique maximum at $s=\frac{2}{3}r$ we also find
\begin{equation*}
    \as^r(K) \leq \as^r(r_K B^2) \leq \as^r\Bigl(\frac{2r}{3}B^2\Bigr),
\end{equation*}
which gives exactly \eqref{eqn:AAC_ineq}.
Hence \eqref{eqn:R-affine-iso} is strictly stronger than both known inequalities in the plane, but without the equality characterisation of \eqref{eqn:AAC_ineq}. It remains an open question if an extension of \eqref{eqn:R-affine-iso} holds true also for $d\geq 3$.

\subsection{Sylvester probabilities in the plane for four points}

Theorem \ref{thm:R-groemer} combined with \eqref{eqn:R-efron} yields
\begin{equation*}
    \EE f_0([K]_4^R) \leq \EE f_0([r_K B^2]_4^R),
\end{equation*}
with equality if and only if $K$ is a disc.

However, unlike in the case of the affine convex hull \eqref{eqn:p44-aff}, the extremisers of $\EE f_0([K]_4^R)$, respectively $\EE\vol_2([K]_3^R)$, do not directly determine extremisers of $p_{4,4}^R(K)$. The Efron--Buchta identities yield
\begin{equation*}
    1-p_{4,4}^R(K)  = 4 \frac{\EE\vol_2([K]_3^R)}{\vol_2(K)} - 6 \frac{\EE V_R(L_K)^2}{\vol_2(K)^2}.
\end{equation*}
The terms $\EE\vol_2([K]_3^R)$ and $\EE V_R(L_K)^2$ are minimised by $r_K B^2$, but appear with opposite signs.
We believe $\EE\vol_2([K]_3^R)$ should dominate the behaviour.

\begin{conjecture}[the disc as extremal body at a fixed area]\label{conj:p43-p44}
    Among all $r$-ball-convex bodies $K\subset\RR^2$ of given area, the disc is the unique maximiser of $p_{4,4}^R(K)$ for all $R\geq r$.
\end{conjecture}

The decomposition of Theorem~\ref{thm:sylvester-general} localises the
conjectured inequality into one term.
For $0\leq t\leq R$ set
\begin{align}\notag
    h_R(t)
    : & = 48\,\Delta_R(t)\,V_R(t) + 3\,V_R(t)^2 - 64\,t^4           \\
      & = 72 t R^2 h \gamma - 84 t^2 R^2 + 20 t^4 + 12 R^4\gamma^2,
    \label{eqn:h-def}
\end{align}
where $h:=\sqrt{R^2-t^2}$ and $\gamma:=\arcsin(t/R)$.
Then for an $r$-ball-convex body $K$ the second
identity of Theorem~\ref{thm:sylvester-general} combined with $p_{4,2}^R(K) = 6 \frac{\EE V_R(L_K)^2}{\vol_2(K)^2}$ gives
\begin{equation}\label{eqn:p43-decomp}
    1-p_{4,4}^R(K)
    = 4\frac{\EE\vol_2\bigl([K]_3\bigr)}{\vol_2(K)}
    + 6\frac{\EE V_R(L_K)}{\vol_2(K)}
    + \frac{\EE h_R(L_K)}{\vol_2(K)^2}.
\end{equation}
For fixed area the first and second terms are minimised by the disc, with equality
precisely for ellipses for the first term and the disc for the second summand.
The second and third terms are functionals of the law of $L_K$ alone and
one might hope to treat both cases separately by the rearrangement of \cite{Pfiefer:1990}. However, $t\mapsto h_R(t)$ is strictly decreasing, and so Pfiefer's argument actually shows that the third summand is maximised by the disc. This is quantified by the following lemma.

\begin{lemma}\label{lem:kernel}
    Let $R>0$. For $t\in (0,R)$,
    \begin{equation}\label{eqn:h-derivative}
        h_R'(t)
        = 24\,V_R(t)\,\frac{2R^2-3t^2}{h}- 64\,t^3
        < 0.
    \end{equation}
    So $h_R$ is strictly decreasing, with
    $h_R(t)=-\tfrac{112}{15R^2}\,t^6+O(t^8)$ as $t\downarrow0$.
\end{lemma}

\begin{proof}
    The derivative \eqref{eqn:h-derivative} follows from \eqref{eqn:segment-2d}, \eqref{eqn:V-derivative} and
    $\Delta_R(t)V_R'(t)=4t^3$. If $2R^2-3t^2 \leq 0$, both summands of
    \eqref{eqn:h-derivative} are non-positive and the second is strictly
    negative.  If $0<2R^2-3t^2$, then, since the integrand of
    $V_R(t)=\int_0^t 4s^2(R^2-s^2)^{-1/2}\,\dint s$ has a decreasing
    denominator,
    \begin{equation*}%\label{eqn:V-upper}
        V_R(t)\;<\;\frac{4\,t^3}{3\sqrt{R^2-t^2}}\,,
    \end{equation*}
    so the first term of \eqref{eqn:h-derivative} is smaller than
    $32t^3(2R^2-3t^2)/(R^2-t^2)\leq64t^3$ since
    $2R^2-3t^2\leq2(R^2-t^2)$.  For the expansion insert
    $\Delta_RV_R=\tfrac43t^4-\tfrac{4}{15R^2}t^6-\tfrac{16}{105R^4}t^8+O(t^{10})$
    and $V_R^2=\tfrac{16}{9R^2}t^6+\tfrac{16}{15R^4}t^8+O(t^{10})$ into
    \eqref{eqn:h-def} and note that the quartic terms cancel identically.
\end{proof}

Since the second and third terms of \eqref{eqn:p43-decomp} depend only on the law of $L_K$,
we may look at them combined and see when we may still apply Pfiefer's argument.
For $0<t<R$, define
\begin{equation}\label{eqn:Amin-def}
    A_{\min}^R(t)
    := -\,\frac{h_R'(t)\,h}{24\,t^2}
    = \frac{8}{3}\,\Delta_R(t)-2\Bigl(\frac{R^2}{t^2}-\frac{3}{2}\Bigr)\,V_R(t),
\end{equation}
so that, for every $A>0$,
\begin{equation}\label{eqn:g-derivative}
    \bigl(6A\,V_R+h_R\bigr)'(t)
    =\frac{24\,t^2}{h}\,\bigl(A-A_{\min}^R(t)\bigr),
\end{equation}
and the combined kernel is non-decreasing on an interval $(0,\tau]$ exactly
when $A$ dominates $A_{\min}^R$ there.

\begin{lemma}[the threshold function]\label{lem:Amin}
 We have the following:
    \begin{enumerate}
        \item[(a)] $A_{\min}^R(t)=R^2A_{\min}^1(t/R)$ and
              \begin{equation*}
                  A_{\min}^1(t)=\tfrac{28}{15}\,t^3+O(t^5)\ \text{ as }
                  t\downarrow0,\qquad
                  \lim_{t\uparrow1}A_{\min}^1(t)=\pi .
              \end{equation*}
        \item[(b)] $A_{\min}^1$ is strictly increasing on $(0,1)$, with
              \begin{equation}\label{eqn:Amin-derivative}
                  \bigl(A_{\min}^1\bigr)'(t)
                  = \frac{4\,V_1(t)}{t^3}-\frac{16-20t^2}{3\sqrt{1-t^2}}
                  > 0 .
              \end{equation}
        \item[(c)] $t\mapsto A_{\min}^1(t)/t^3$ is strictly increasing on
              $(0,1)$. In particular
              \begin{equation}\label{eqn:Amin-cubic}
                  A_{\min}^1(t)
                  \leq \pi\, t^3
                  \qquad\text{for }0<t<1.
              \end{equation}
    \end{enumerate}
\end{lemma}

\begin{proof}
    (a) The scaling follows from $\Delta_R(t)=R^2\Delta_1(t/R)$ and
    $V_R(t)=R^2V_1(t/R)$, the expansion from Lemma~\ref{lem:kernel}.
    The limit follows from
    $V_1(1)=\pi$ and $\Delta_1(1)=0$.

    (b) Differentiating \eqref{eqn:Amin-def} at $R=1$, the two $V_1$-terms
    combine, $\tfrac{10\,V_1}{t}+\tfrac{2(2-3t^2)V_1}{t^3}=\tfrac{4(1+t^2)\,V_1}{t^3}$,
    and the remaining algebraic terms give the second term of
    \eqref{eqn:Amin-derivative}.  Positivity is equivalent to
    $12\,V_1(t)\sqrt{1-t^2}>4(4-5t^2)\,t^3$.  If $5t^2\geq4$ the
    right-hand side is non-positive. Otherwise, $V_1(t)>\tfrac43t^3$, by
    integrating $V_1'(t)\geq4t^2$, and
    $\sqrt{1-t^2}\geq1-t^2>1-\tfrac{5}{4}t^2$ give the claim.

    (c)
    To show that $t\mapsto t^{-3}\, A_{\min}^1(t)$ is strictly increasing, we consider the derivative and claim that
    $t\bigl(A_{\min}^1\bigr)'(t)-3\,A_{\min}^1(t)>0$. By \eqref{eqn:Amin-def} and \eqref{eqn:Amin-derivative}
    \begin{equation*}
        t\bigl(A_{\min}^1\bigr)'(t)-3\,A_{\min}^1(t)
        = \frac{(10-9t^2)\,V_1(t)}{t^2}
        -\frac{4\,t\,(10-11t^2)}{3\sqrt{1-t^2}}.
    \end{equation*}
    If $11t^2\geq 10$, then the second term is non-positive and the first term is strictly positive, and the claim follows. So, assume $11t^2< 10$.
    It suffices to prove $3\,(10-9t^2)\,V_1(t)\sqrt{1-t^2}>4\,t^3(10-11t^2)$.
    Since $V_1'(t)=4t^2(1-t^2)^{-1/2}$ has non-negative
    Taylor coefficients, $V_1(t)\geq\tfrac43t^3+\tfrac25t^5+\tfrac3{14}t^7$, and inserting
    this bound reduces the claim, with $x:=t^2\in(0,\tfrac{10}{11})$, to
    \begin{equation*}
        (10-9x)\, \Bigl(1+\tfrac{3}{10}x+\tfrac{9}{56}x^2\Bigr)\, \sqrt{1-x} \geq 10-11x.
    \end{equation*}
    Both sides are positive in this range, so squaring is an equivalence, and
    \begin{equation*}
        (10-9x)^2\, \Bigl(1+\frac{3}{10}x+\frac{9}{56}x^2\Bigr)^2\, (1-x)-\bigl(10-11x\bigr)^2
        =x^2 \, \Phi(x),
    \end{equation*}
    where
    \[
        \Phi(x)=\frac{92}{7}-\frac{2097}{70}x+\frac{673569}{19600}x^2-\frac{1539}{100}x^3-\frac{16767}{15680}x^4-\frac{6561}{3136}x^5.
    \]
    Since for $x\in [0,\frac{10}{11}]$, $x^3\leq \frac{10}{11}x^2$, $x^4\leq x^2$ and $x^5\leq x^2$, we find
    \begin{equation*}
        \Phi(x) \geq \frac{92}{7}-\frac{2097}{70}\,x+\frac{132543}{7700}\,x^2 > 0.
    \end{equation*}
    Finally,
    $t^{-3}A_{\min}^1(t)$ increases from $\tfrac{28}{15}$ to
    $\lim_{t\uparrow1}t^{-3} A_{\min}^1(t)=\pi$ by (a), which is
    \eqref{eqn:Amin-cubic}.
\end{proof}

\begin{proposition}\label{prop:criterion}
    Let $K\subset\RR^2$ be an $r$-ball-convex body and let $R\geq r$.  If
    \begin{equation}\label{eqn:criterion}
        \vol_2(K)\;\geq\;A_{\min}^R\bigl(\tfrac{\diam K}{2}\bigr),
    \end{equation}
    then,
    \begin{equation*}
        p_{4,4}^R(K) \leq p_{4,4}^{R}(r_KB^2),
    \end{equation*}
    with equality if and only if $K$ is a disc.
\end{proposition}

\begin{proof}
    Set $D:=\diam K$. If $D=2r$, then $K=rB^2$ and the claim holds true.
    So assume $D<2r$ and set
    \[
        \Psi(t):=6\, \vol_2(K)\,V_R(t) + h_R(t), \qquad \text{for $t\in[0,D/2]$.}
    \]

    Note that $L_K\leq D/2$, and by the isodiametric
    inequality $\pi r_K^2 = \vol_2(K)\leq\pi(D/2)^2$, see \cite{Sch14}, we also have  $L_{r_K B^2}\leq r_K\leq D/2$.
    Since $D/2<r\leq R$, the kernel $\Psi$ is continuously differentiable on
    $[0,D/2]$, and \eqref{eqn:g-derivative}, the strict monotonicity of
    $A_{\min}^R$ from Lemma~\ref{lem:Amin}\,(b) and \eqref{eqn:criterion} give
    \begin{equation*}
        \Psi'(t) = \frac{24\,t^2}{\sqrt{R^2-t^2}}\,
        \bigl(\vol_2(K)-A_{\min}^R(t)\bigr) > 0
        \qquad\text{for every }t\in\bigl(0,\tfrac{D}{2}\bigr),
    \end{equation*}
    so that $\Psi$ is strictly increasing on $[0,D/2]$. Pfiefer's theorem \cite{Pfiefer:1990} therefore applies to it and yields
    \begin{equation}\label{eqn:kernel-gap}
        \EE\,\Psi(L_K) \geq \EE\,\Psi(L_{r_K B^2}),
    \end{equation}
    with equality only if $K$ is a disc.
    The decomposition \eqref{eqn:p43-decomp} yields
    \[
        1-p_{4,4}^R(K) = 4 \frac{\EE\vol_2([K]_3)}{\vol_2(K)} + \frac{\EE \Psi(L_K)}{\vol_2(K)^2}.
    \]
    Thus \eqref{eqn:kernel-gap} together with Groemer's inequality
    \cite{Groemer:1974} for the first term, yields
    $p_{4,4}^R(K)\leq p_{4,4}^R(r_KB^2)$.
    Equality for $p_{4,4}^R$ forces equality in \eqref{eqn:kernel-gap} and hence $K$ is a disc.
\end{proof}

\begin{proof}[Proof of Theorem~\ref{thm:intro-p43}]
    Write $D:=\diam K$ and assume that
    \[
        D^3 \leq \frac{8}{\pi} \vol_2(K)\, R.
    \]
    It suffices to verify the criterion \eqref{eqn:criterion}
    and to apply Proposition~\ref{prop:criterion}, which also gives the
    equality case.
    Since $K$ is $r$-ball-convex we have $D\leq 2r\leq 2R$. Thus, by
    Lemma~\ref{lem:Amin},
    \begin{equation*}
        A_{\min}^R\Bigl(\frac D2\Bigr)
        = R^2 A_{\min}^1\Bigl(\frac D{2R}\Bigr)
        \leq \pi R^2\, \Bigl(\frac{D}{2R}\Bigr)^3 = \frac{\pi D^3}{8R} \leq \vol_2(K),
    \end{equation*}
    which is \eqref{eqn:criterion}.

    Assume now that $K$ is a convex body of constant width $r$.  Such a body
    is $r$-ball-convex, see
    \cite{BLNP:2007}, hence $R$-ball-convex by
    Lemma~\ref{lem:elementary}\,(c), and satisfies $D=r$. By the
    Blaschke--Lebesgue theorem its area is at least
    that of the Reuleaux triangle of the same width, that is,
    $\vol_2(K)\geq\tfrac{\pi-\sqrt3}{2}\,r^2$, see \cite{Sch14}. Hence
    \begin{equation*}
        \frac{8}{\pi} \, \vol_2(K)\, R
        \geq \frac{4(\pi-\sqrt{3})}{\pi} r^3
        > r^3 = D^3.
    \end{equation*}
    Thus \eqref{eqn:small-body} holds true.
\end{proof}

We note that \eqref{eqn:criterion} holds true for any disc $K=sB^2$ for $s\in [0,R]$, since then $\diam K = 2s$ and by \eqref{eqn:Amin-cubic}
\[
    A_{\min}^R(s) \leq \pi \frac{s^3}{R} \leq s^2\pi = \vol_2(K),
\]
with equality exactly for $R=s$.

However, if $K=\mathrm{Sg}_R(\ell)$ is an $R$-ball-convex segment with $0<\ell<R$, then $\diam K = 2\ell$, $\vol_2(K) = V_R(\ell)$ and we find
\begin{equation*}
    A_{\min}^R(\ell) - \vol_2(K)
    = \frac{8}{3} \, \ell h - \frac{2h^2\, V_R(\ell)}{\ell^2}
    = \frac{2h}{\ell^2}\, \Bigl( \frac 43\ell^3 - h\, V_R(\ell)\Bigr) >0
\end{equation*}
since $\psi(\ell):= \frac{4}{3} \ell^3 - h\, V_R(\ell)$ satisfies $\psi(0)=0$ and $\psi'(\ell) = \ell V_R(\ell)/h >0$.

\section{Closing discussion}\label{sec:discussion}

Figure~\ref{fig:pnk} displays the values of $p_{n,k}^R(B^2)$, $n=3,4$, as
functions of $R$, and Figure~\ref{fig:evol} shows the corresponding expected
areas together with the second moment $m_2(R)$.

\begin{figure}[t]
    \centering
    \includegraphics[width=0.48\linewidth]{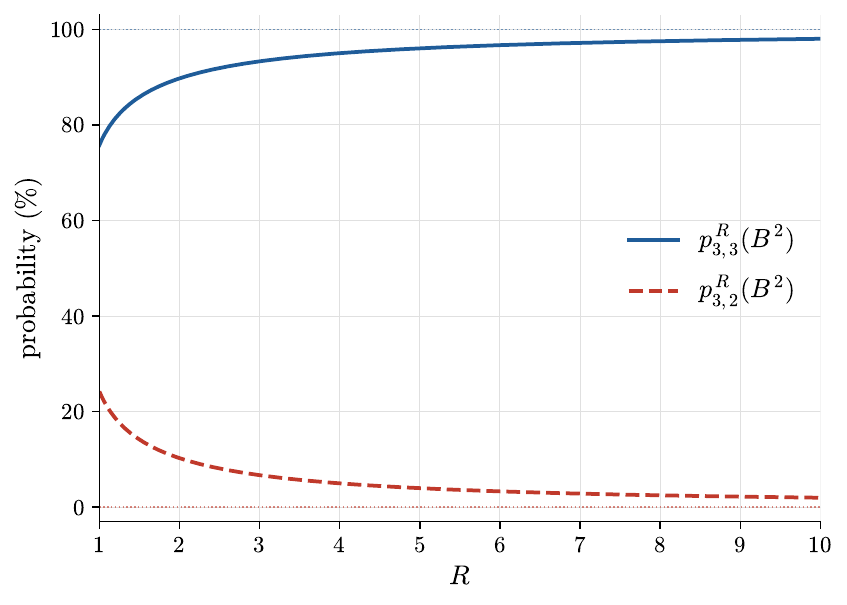}\hfill
    \includegraphics[width=0.48\linewidth]{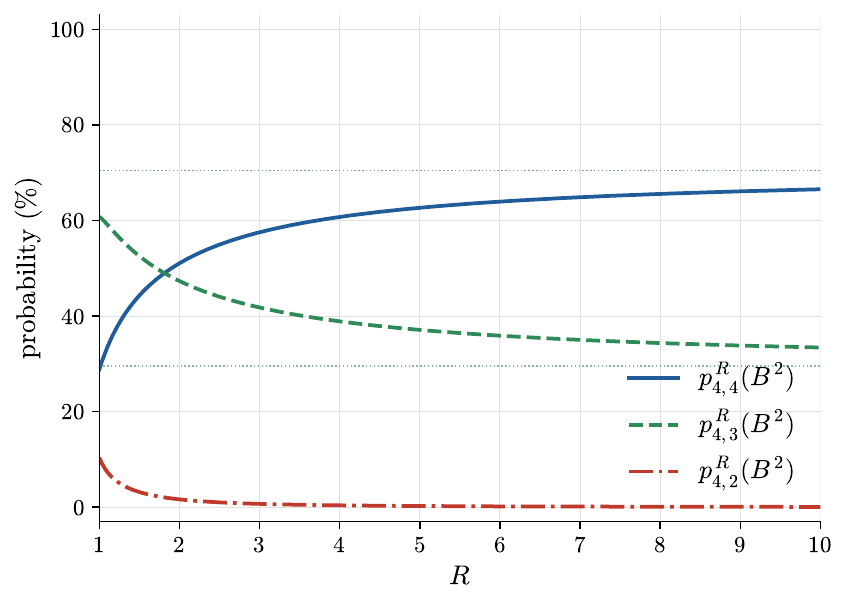}
    \caption{Left: the probabilities $p_{3,3}^R(B^2)$ and
        $p_{3,2}^R(B^2)=\PP(\varrho([B^2]_3)>R)$ as functions of $R\geq1$,
        computed from Corollary~\ref{cor:intro-circumradius}; the dotted
        lines mark the limits $100\%$ and $0\%$ as $R\to\infty$.
        Right: the probabilities $p_{4,k}^R(B^2)$, $k=2,3,4$, computed from
        \eqref{eqn:vol3-explicit}; the dotted lines mark the limits
        $p_{4,4}^R\to 1-\tfrac{35}{12\pi^2}$ and
        $p_{4,3}^R\to\tfrac{35}{12\pi^2}$.}
    \label{fig:pnk}
\end{figure}

\begin{figure}[t]
    \centering
    \includegraphics[width=0.48\linewidth]{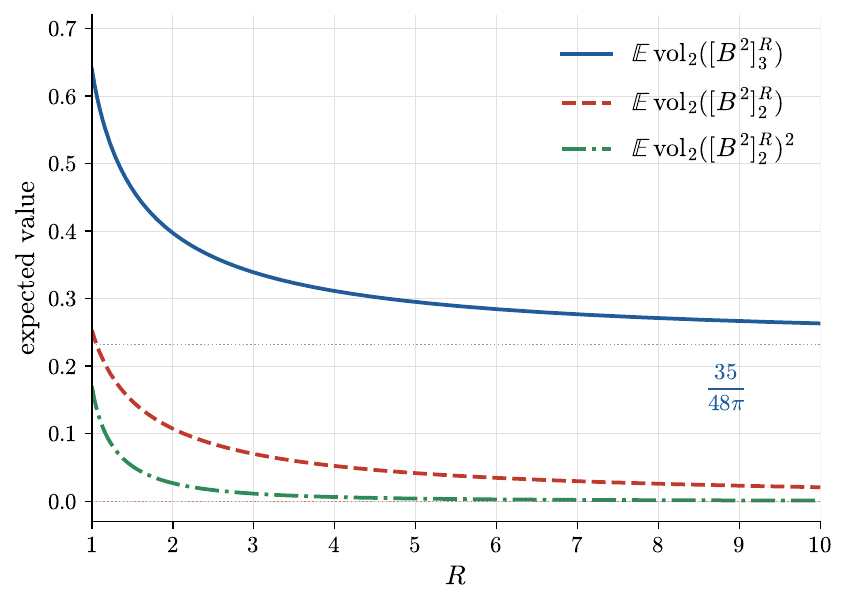}
    \caption{The expected areas $\EE\vol_2([B^2]_2^R)=m_1(R)$ and
        $\EE\vol_2([B^2]_3^R)$, together with the second moment
        $\EE\vol_2([B^2]_2^R)^2=m_2(R)$, as functions of $R\geq1$, from
        \eqref{eqn:k-moment} and \eqref{eqn:vol3-explicit}.  As $R\to\infty$,
        $\EE\vol_2([B^2]_2^R)\to0$ and $\EE\vol_2([B^2]_2^R)^2\to0$, at the
        rates $m_1\sim\tfrac{1024}{1575\pi R}$ and $m_2\sim\tfrac{7}{72R^2}$ of
        Remark~\ref{rem:asymptotics}, while
        $\EE\vol_2([B^2]_3^R)\to\EE\vol_2([B^2]_3)=\tfrac{35}{48\pi}$
        (dotted lines).}
    \label{fig:evol}
\end{figure}

\begin{question}[monotonicity in $R\mapsto p_{4,3}^R(K)$]\label{qu:monotone-p43}
    The map $R\mapsto\PP(f_0([K]_n^R)\leq k)$ is
    non-increasing for every $k$, so that $p_{n,2}^R(K)$ of \eqref{eqn:pn2} is
    non-increasing and $p_{n,n}^R(K)$ of \eqref{eqn:pnn} is non-decreasing.  For
    $2<k<n$ that argument gives nothing about $p_{n,k}^R(K)$ as it is a difference of
    two non-increasing functions of $R$. Since $R\mapsto p_{4,3}^R(B^2)$
    is decreasing, one question is if $R\mapsto p_{4,3}^R(K)$ is monotone on $[r,\infty)$ for
    every $r$-ball-convex body $K$ or is the observed
    monotonicity particular to the disc?
\end{question}

\begin{question}[extremal values of $p_{4,4}^1$]\label{qu:extremal}
    Blaschke's theorem identifies ellipses as the maximisers and triangles as
    the minimisers of $p_{4,4}(K)$.  The $R$-ball-convex analogue retains only
    similarity invariance, $p_{4,4}^{\lambda R}(\lambda K)=p_{4,4}^{R}(K)$ by
    \eqref{eqn:equivariance}, and the maximisation degenerates: shrinking the
    disc gives $p_{4,4}^1(\lambda B^2)=p_{4,4}^{1/\lambda}(B^2)\to
        p_{4,4}(B^2)$ as $\lambda\downarrow0$, while
    $p_{4,4}^1(K)\leq p_{4,4}(K)\leq p_{4,4}(B^2)$ for every $1$-ball-convex
    $K$ by Blaschke's theorem, so the
    supremum of $p_{4,4}^1$ equals the classical Sylvester probability
    $p_{4,4}(B^2)=1-\frac{35}{12\pi^2}$ and is not attained.  For the minimum
    we believe that for every $1$-ball-convex body $K\subset\RR^2$,
    \begin{equation*}
        p_{4,4}^{1}(K) \;\overset{?}{\geq}\; p_{4,4}^{1}(B^2)
        = 5-\frac{93}{2\pi^2} \approx 28.86\%,
    \end{equation*}
    with equality only for the disc.  Together with
    Conjecture~\ref{conj:p43-p44} this would yield
    \begin{equation*}
        p_{4,4}^1(B^2) \;\leq\; p_{4,4}^1(K) \;\leq\; p_{4,4}^1(r_K B^2).
    \end{equation*}
    Note that the disc is the unique $1$-ball-convex body of area $\pi$, as
    it equals the unit disc containing it, so the conjectured minimality is a
    fixed-area phenomenon.
\end{question}

Corollary~\ref{cor:isoperimetric} rests on the rearrangement theorem of Pfiefer
\cite{Pfiefer:1990}, which has no counterpart in the opposite direction: over
all planar convex bodies of given area $\EE L_K$ is unbounded.  Within the
class of $r$-ball-convex bodies, where $\diam K\leq2r$ forces $\EE L_K$ to be finite,
the natural candidate for the maximiser at a given area is the $r$-ball-segment,
and deterministically it is extremal in exactly this direction.  At a given
area it is the unique maximiser of the diameter, since
$\conv_r\{\bx,\by\}\subset K$ for any two points $\bx,\by$ realising
$\diam K$. It is the unique maximiser of the perimeter, by the reverse
isoperimetric inequality of Borisenko and Drach \cite{BD:2014}, see also
\cite{FKV:2016}. It is also the unique maximiser of the circumradius in a given
volume in every dimension, according to a theorem of Bezdek \cite{Bezdek:2021}.

\begin{figure}[t]
    \centering
    \includegraphics[width=0.65\linewidth]{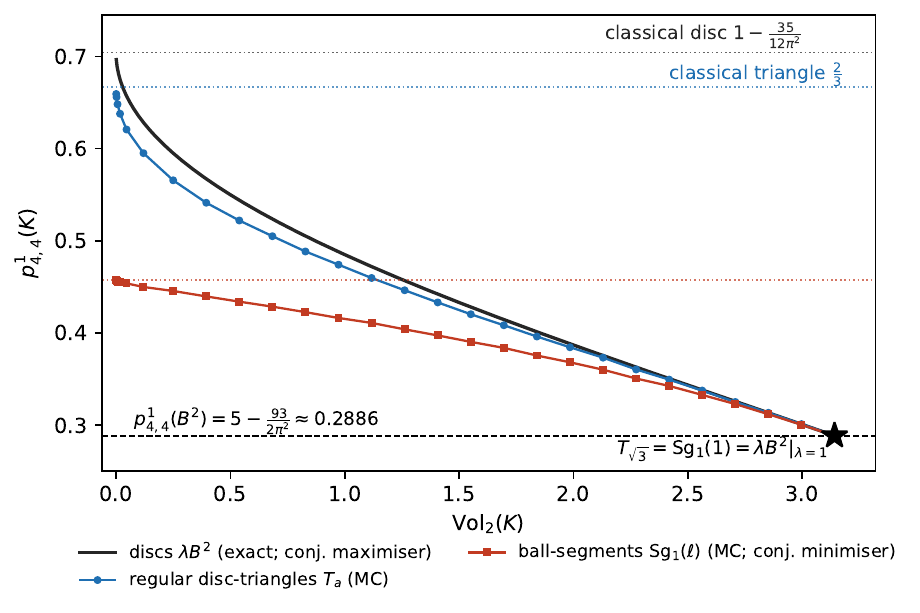}
    \caption{The probability $p_{4,4}^1(K)$ as a function of the area
    $\vol_2(K)\in(0,\pi]$ along three families of $1$-ball-convex bodies:
    the ball-segments $\mathrm{Sg}_1(\ell)$ and the disc-triangles
    $T_a:=\conv_1\{\ba_1,\ba_2,\ba_3\}$ spanned by the vertices of a
    regular triangle of side $a\in(0,\sqrt3]$, for which the values are
    Monte Carlo estimates from $N=2\times10^6$ samples per body
    (standard error $\leq4\times10^{-4}$), and the
    discs $\lambda B^2$, $\lambda\in(0,1]$.  The solid black curve is the
    exact value for the discs, computed from
    $p_{4,4}^1(\lambda B^2)=p_{4,4}^{1/\lambda}(B^2)$ and
    \eqref{eqn:vol3-explicit}.
    }
    \label{fig:p44-families}
\end{figure}

\begin{question}[the ball-segment as extremal body at given area]%\label{qu:segment}
    Among all $1$-ball-convex bodies $K\subset\RR^2$ of given area, does the
    $1$-ball-segment of that area maximise $\EE\vol_2\left([K]_2^R\right)^k$ for
    all $k\geq1$ and $R\geq 1$? By \eqref{eqn:pn2}, this would complement Corollary~\ref{cor:isoperimetric} exactly, that is,
    \begin{equation*}
        p_{n,2}^R(r_KB^2) \;\leq\; p_{n,2}^R(K) \overset{?}{\leq} p_{n,2}^R(\mathrm{Sg}_1(\ell_K)),
    \end{equation*}
    where $\ell_K\in [r_K,1]$ is uniquely determined by $\vol_2(\mathrm{Sg}_1(\ell_K))=\vol_2(K)$.
    For
    $p_{3,2}^R$ and $p_{4,2}^R$ only the cases $k=1,2$ are needed.
    The
    strongest form of the question is stochastic: is the half-distance $L_K$ of
    two independent uniform random points in $K$ stochastically dominated by
    that of the $1$-ball-segment of equal area?  Both endpoints of the
    distribution are settled in the ball-segment's favour: near the right endpoint
    the ball-segment uniquely maximises the diameter, and as $t\downarrow0$ the
    expansion
    $\PP(2L_K\leq t)=\pi t^2/\vol_2(K)
        -\tfrac{2}{3}\operatorname{per}(K)\,t^3/\vol_2(K)^2+O(t^4)$,
    together with the reverse isoperimetric inequality \cite{BD:2014}, shows
    that the ball-segment has the pointwise smallest distribution function to
    the third order at the left endpoint.  Pfiefer's
    two-point rearrangement is unavailable here, since the maximising
    rearrangement degenerates without the curvature constraint. The
    rearrangement techniques behind \cite{BD:2014} and the sharp mean-distance
    bounds of Bonnet, Gusakova, Th\"ale and Zaporozhets \cite{BGTZ:2021}
    suggest themselves instead.

    We believe that the ball-segment mirrors the disc of
    Conjecture~\ref{conj:p43-p44} in the entire vertex distribution, see Figure \ref{fig:p44-families}.
    So we also ask: does the $1$-ball-segment of given area minimise
    $p_{4,4}^R(K)$, that is,
    \begin{equation*}
        p_{4,4}^R(\mathrm{Sg}_1(\ell_K)) \overset{?}{\leq} p_{4,4}^R(K)
        \;\overset{?}{\leq}\; p_{4,4}^R(r_KB^2),
    \end{equation*}
    the right-hand inequality being Conjecture~\ref{conj:p43-p44}.  The
    family of ball-segments terminates at the disc:
    $\mathrm{Sg}_1(1)=B^2$ and
    $\ell\mapsto\vol_2(\mathrm{Sg}_1(\ell))$ increases from $0$ to $\pi$ on
    $(0,1]$.  A positive answer and if $\ell\mapsto p_{4,4}^1(\mathrm{Sg}_1(\ell))$
    is monotonically decreasing from $\lim_{\ell\downarrow 0}p_{4,4}^1(\mathrm{Sg}_1(\ell))$ to $p_{4,4}^1(B^2)$,
    would therefore prove the lower bound of Question~\ref{qu:extremal}.
\end{question}

\begin{remark}
    We note that using the $R$-ball-convex Efron--Buchta identities is only one of the possible methods to determine Sylvester probabilities. In fact, one can also use Blaschke--Petkantschin formulas for spheres, such as Theorem~7.3.1 in \cite{SW:2008}. However, it seems that this transformation is effective only in the case when $K=rB^2$. Another possible approach could be via Crofton's theorem, which has been used successfully for several classical problems before. It remains to be seen whether this technique can be adapted to the $R$-ball-convex Sylvester setting.
\end{remark}

\section*{Funding}
This research was supported by NKFIH project no.~150151. Project no.~150151 has been implemented with the support provided by the Ministry of Culture and Innovation of Hungary from the National Research, Development and Innovation Fund, financed under the ADVANCED\_24 funding scheme.

\section*{AI-usage disclosure}
During the preparation of this work Claude (Anthropic) and ChatGPT (OpenAI) were used for mathematical discussions and editorial assistance. In particular, these tools helped check and simplify the closed form expressions for $m_1(R),m_2(R),\mu(R)$ as presented in Section \ref{sec:moments} and pointed out the connection to the numerical results in \cite{LeCaer:2017} for the circumradius distribution.
They were also used to generate scripts to run Monte Carlo checks for the statements and claims in this article, in particular for Figure \ref{fig:p44-families}.
All mathematical claims and proofs were independently
checked, revised, and written in their final form by the authors.

\raggedbottom

\end{document}